\documentclass[reqno,11pt,a4paper]{amsart}
\usepackage[foot]{amsaddr}
\usepackage[a4paper,left=2cm,right=2cm,top=2cm,bottom=2cm]{geometry}

\usepackage{titlesec}

\newcommand*{\justifyheading}{\raggedright}
\titleformat {\title}
{\normalfont \HUGE \bfseries \justifyheading}{\thesection}{1em}{}
\titleformat {\section}
{\normalfont \Large \bfseries \justifyheading}{\thesection}{1em}{}
\titleformat {\subsection}
{\medskip \normalfont \large \bfseries \justifyheading}{\thesubsection}{1em}{}
\titleformat {\subsubsection}
{\smallskip \normalfont \bfseries \justifyheading}{\thesubsubsection}{1em}{}

\usepackage{color}

\usepackage[backend=bibtex,bibencoding=ascii,giveninits=true,url=false,isbn=false,doi=false,style=numeric,maxnames=5]{biblatex}
\usepackage[utf8]{inputenc}
\usepackage{graphicx}
\usepackage{epstopdf}

\usepackage{booktabs}

\usepackage{tensor}
\usepackage{amsmath}
\usepackage{amsthm}
\usepackage{amsfonts}
\usepackage{amssymb}
\usepackage{nicefrac}
\usepackage{dsfont}
\usepackage[inline]{enumitem}

\usepackage{tikz}

\usepackage[
bookmarks,
bookmarksopen,
bookmarksopenlevel=0,
bookmarksnumbered,
breaklinks,
ocgcolorlinks,
colorlinks=true,
linkcolor=linkcolor,
citecolor=citecolor,
urlcolor=urlcolor,
linktoc=all
]{hyperref}

\hypersetup{
	pdftitle    = {Levi-Civita metrics that admit a projective vector field},
	pdfauthor   = {Gianni Manno, Andreas Vollmer},
	pdfsubject  = {Classification of projective connections and superintegrability}
}

\definecolor{linkcolor}{rgb}{0.5,0.0,0.0}
\definecolor{citecolor}{rgb}{0.0,0.5,0.0}
\definecolor{urlcolor} {rgb}{0.0,0.0,0.5}

\theoremstyle{plain}

\newtheorem{lemma}{Lemma}
\newtheorem{theorem}{Theorem}
\newtheorem{proposition}{Proposition}
\newtheorem{corollary}{Corollary}

\theoremstyle{definition}
\newtheorem{definition}{Definition}

\newtheorem{example}{Example}

\theoremstyle{definition}
\newtheorem{remark}{Remark}

\newcommand{\projalg}{\ensuremath{\mathfrak{p}}}
\newcommand{\solSp}{\ensuremath{\mathfrak{A}}}

\newcommand{\lie}{\ensuremath{\mathcal{L}}}
\newcommand{\solpol}{\ensuremath{\mathrm{S}}}

\newcommand{\R}{\ensuremath{\mathds{R}}}

\newcommand{\lb}{\ensuremath{\left(}}
\newcommand{\rb}{\ensuremath{\right)}}
\newcommand{\la}{\ensuremath{\left\langle}}
\newcommand{\ra}{\ensuremath{\right\rangle}}

\DeclareMathOperator{\tr}{tr}
\DeclareMathOperator{\trace}{trace}

\DeclareMathOperator{\rank}{rank}
\DeclareMathOperator{\Id}{Id}

\newcommand{\weg}[1]{}
\newcommand{\bq}{\begin{equation}}
\newcommand{\eq}{\end{equation}}

\makeatletter
\newcommand{\allowhy}{\nobreak\hskip\z@skip}
\makeatother

\makeatletter
\newcommand{\pushright}[1]{\ifmeasuring@#1\else\omit\hfill$\displaystyle#1$\fi\ignorespaces}
\newcommand{\pushleft}[1]{\ifmeasuring@#1\else\omit$\displaystyle#1$\hfill\fi\ignorespaces}
\makeatother

\begin{document}

\title{3-dimensional Levi-Civita metrics with projective vector fields}
\author{Gianni Manno$^\star$}
\email{$^\star$giovanni.manno@polito.it}
\address{\upshape
    Dipartimento di Scienze Matematiche,
    Politecnico di Torino,
    Corso Duca degli Abruzzi, 24,
    10129 Torino, Italy
}

\author{Andreas Vollmer$^\mathsection$}
\email{$^\mathsection$andreas.vollmer@polito.it}

\makeatletter
\@namedef{subjclassname@2020}{%
	\textup{2020} Mathematics Subject Classification}
\makeatother

\begin{abstract}
 Projective vector fields are the infinitesimal transformations whose local
 flow preserves geodesics up to reparametrisation.
 In 1882 
 Sophus Lie posed the problem of describing 2-dimensional metrics admitting a
 non-trivial projective vector field, which was solved in recent years.
 In the present paper, we solve the analog of Lie's problem in dimension 3,
 for Riemannian metrics and, more generally, for Levi-Civita metrics of
 arbitrary signature.
\end{abstract}
\maketitle

\setcounter{tocdepth}{2}

\section{Introduction}

\subsection{Basic definitions and description of the problem}

Let $M$ be a smooth manifold of dimension $n$.
A \emph{metric} on $M$ is a symmetric, non-degenerate $(0,2)$-tensor field
$g$ (of arbitrary signature). In particular, if $g$ is positive
(resp.~negative) definite, it is called a \emph{Riemannian}
(resp.~\emph{anti-Riemannian}) metric.
If the signature of $g$ is either $(+\dots+-)$ or $(+-\dots-)$, it is called
\emph{Lorentzian}.
A metric is of \emph{constant curvature} if its sectional curvatures coincide
and are constant.
A curve $\R\supseteq I\ni t\mapsto\gamma(t)\in M$ such that
\begin{equation}\label{eqn:geodesic.equation}
	\nabla_{\dot\gamma}\dot\gamma = f(t)\dot\gamma
\end{equation}
for some function $f(t)$ is an \emph{unparametrised geodesic of $g$}.

\smallskip\noindent
The present paper studies metrics that admit vector fields whose local
flow preserves unparametrised geodesics.

\begin{definition}
	A \emph{projective transformation} is a (local) diffeomorphism of $M$
	that sends geodesics into geodesics (where geodesics are to be understood
	as unparametrised curves).
	A vector field on $M$ is \emph{projective} if its (local) flow acts by
	projective transformations.
	A projective vector field $v$ is \emph{homothetic} (for the metric $g$)
	if the Lie
	derivative of $g$ along $v$ satisfies $\lie_vg=\lambda g$ for some
	constant
	$\lambda\in\R$.  If $\lambda\neq0$, we say that $v$ is \emph{properly
	homothetic}.
	If $v$ is not homothetic, we call it an \emph{essential} projective
	vector field.
\end{definition}

\noindent The projective vector fields of a given metric $g$ form a Lie
algebra \cite{lie_1883}, which we denote by $\projalg(g)$.
It is easy to confirm that the following definition indeed gives rise to an
equivalence relation.

\begin{definition}
	We say that two metrics (on the same manifold) are \emph{projectively
	equivalent} if they share the same geodesics (as
	unparametrised curves). The set of all  metrics projectively equivalent
	to a given metric~$g$ is called the \emph{projective class} of~$g$.
\end{definition}
\smallskip

In the 1880s, Sophus Lie posed the following problem for 2-dimensional
surfaces
\cite{lie_1882,lie_1883}: \emph{Find the metrics that describe surfaces whose
geodesic
	curves admit an infinitesimal transformation}\footnote{Original German
	wording in \cite{lie_1882}: \emph{``Es wird verlangt, die Form des
		Bogenelementes einer jeden Fl\"ache zu bestimmen, deren geod\"atische
		Curven
		eine infinitesimale Transformation gestatten.''}}.
This problem has been solved during the recent years in
\cite{bryant_2008,matveev_2012,manno_2020}.
Related work can be found in
\cite{dini_1869,liouville_1889,fubini_1903,aminova_2003,aminova_2006,bryant_2009}.

The present paper is concerned with the analogous problem for 3-dimensional manifolds, i.e., to find 3-dimensional metrics
admitting a non-trivial projective vector field. We solve this problem
 for 3-dimensional Riemannian metrics and, more generally, for Levi-Civita
 metrics (thoroughly introduced in Section~\ref{sec:LC.metrics}) of
 arbitrary signature.

%
%

\subsection{State of the art}

It is well-known that, given a 3-dimensional metric $g$ that admits a
homothetic vector field $v$, there exist, around a non-singular point of $v$,
local coordinates $(x,y,z)$ such that $v=\partial_x$ and
\begin{equation}\label{eqn:homothetic.normal.form}
	g = e^{\lambda\,x} h(y,z)\,,\quad\lambda\in\R\,,
\end{equation}
where $h(y,z)$ is a non-degenerate metric
depending on the coordinates $y$ and $z$ only.
It is easy to confirm that, in the case when all metrics in a projective
class $[g]$ are proportional (by a real constant), any projective vector
field is homothetic and thus the metric $g$ locally can be written
in the form~\eqref{eqn:homothetic.normal.form}. For this reason, we focus our
attention to metrics whose projective class contains a pair of non-proportional
metrics.
%

\smallskip\noindent
It follows from ~\cite{levi-civita_1896,matveev_2012rel} that a Riemannian metric $g$ of non-constant curvature and belonging to the aforementioned class
can be taken to be locally either in the form
\begin{equation}\label{eqn:LC.111}
	\pm(X_1-X_2)(X_1-X_3)\,(dx^1)^2\pm(X_2-X_1)(X_2-X_3)\,(dx^2)^2\pm(X_3-X_1)(X_3-X_2)\,(dx^3)^2\,,\,\,X_i=X_i(x^i)
\end{equation}
or in the form
\begin{equation}\label{eqn:LC.21}
	\zeta(x^3)\,(h \pm (dx^3)^2)\,,
\end{equation}
where $h=\sum_{i,j=1}^2h_{ij}dx^idx^j$ is a 2-dimensional metric (note
however, that not all metrics~\eqref{eqn:LC.111} and~\eqref{eqn:LC.21} are
Riemannian).
The metrics~\eqref{eqn:LC.111} and~\eqref{eqn:LC.21} are the so-called
\emph{Levi-Civita metrics}; this class of metrics will be thoroughly
introduced in Section~\ref{sec:LC.metrics}.
Our aim will be to find the functions $X_i(x^i)$ and $\zeta(x^3)$, and the
metrics $h$, such that a projective vector field exists for $g$.
\medskip

Projective vector fields for Levi-Civita metrics, mainly those of the
form~\eqref{eqn:LC.111}, are studied in \cite{solodovnikov_1956}, under the
assumption of Riemannian signature.
The reference then finds explicit ordinary differential equations (ODEs) for
the functions $X_i$ and the projective symmetries, both essential and
homothetic ones.
Generalisations of some of these ODEs can be found in recent works concerned
with splitting-gluing constructions of projectively equivalent metrics
\cite{bolsinov_2011} and c-projective vector fields \cite{bolsinov_2015cproj}
and will be discussed in detail later.

\smallskip

\noindent However, explicit descriptions of 3-dimensional metrics with
projective vector fields do not exist to date except for some special cases.
The current paper closes this gap by providing explicit descriptions of
3-dimensional metrics with projective vector fields.

\subsection{A first look at the results}

We explicitly describe all metrics~\eqref{eqn:LC.111} and~\eqref{eqn:LC.21}
with non-zero projective
vector fields.
The explicit metrics can be found in Sections~\ref{sec:111} and~\ref{sec:21},
see Theorems~\ref{thm:1.1.1} and~\ref{thm:2.1}, respectively.
The following theorem is a first look at the results we obtain, stated for
3-dimensional Riemannian metrics. If the metrics in Theorem~\ref{thm:main}
are of non-Riemannian signature, they still admit the projective vector
fields indicated. We proceed to make this precise in the main body of the
paper.
A final remark concerns the case when the metric assumes the
form~\eqref{eqn:homothetic.normal.form} in suitable local coordinates.
A description of the projective algebras of these metrics will be
given in this paper, in Theorems~\ref{thm:1.1.1} and~\ref{thm:2.1}, along
with Corollary~\ref{cor:constant.Z} and Proposition~\ref{prop:1D.algebra.21},
but is omitted here for brevity.

\begin{theorem}\label{thm:main}
	Let $g$ be a 3-dimensional Riemannian metric with a non-vanishing
	projective vector field. Then either $g$ is locally of the
	form~\eqref{eqn:homothetic.normal.form} or it is among the following:
	\begin{enumerate}
		\item\label{item:main.111.1}
		The metric, for $\beta\ne0, k_i\ne0$,
		\begin{align*}
			g &= k_1(\tanh(x)-\tanh(y))
					(\tanh(x)-\tanh(z))\,e^{\frac{2x}{\beta}}\,dx^2
			\\ & \quad
				+k_2(\tanh(y)-\tanh(x))
					(\tanh(y)-\tanh(z))\,e^{\frac{2y}{\beta}}\,dy^2
			\\ & \quad
				+k_3(\tanh(z)-\tanh(x))
					(\tanh(z)-\tanh(y))\,e^{\frac{2z}{\beta}}\,dz^2
		\end{align*}
		with the essential projective vector field
		$v=\partial_x+\partial_y+\partial_z$.
		\item\label{item:main.111.2}
		The metric, for $k_i\ne0$,
		\begin{align*}
			g &= k_1\left(\tfrac1x-\tfrac1y\right)
					\left(\tfrac1x-\tfrac1z\right)\,e^{2x}\,dx^2
				+k_2\left(\tfrac1y-\tfrac1x\right)
					\left(\tfrac1y-\tfrac1z\right)\,e^{2y}\,dy^2
			\\ & \quad
				+k_3\left(\tfrac1z-\tfrac1x\right)
					\left(\tfrac1z-\tfrac1y\right)\,e^{2z}\,dz^2
		\end{align*}
		with the essential projective vector field
		$v=\partial_x+\partial_y+\partial_z$.
		\item\label{item:main.111.3}
		The metric, for $\beta\ne0, k_i\ne0$,
		\begin{align*}
			g &= k_1(\tan(x)-\tan(y))
					(\tan(x)-\tan(z))\,e^{\frac{2x}{\beta}}\,dx^2
			\\ & \quad
				+k_2(\tan(y)-\tan(x))
					(\tan(y)-\tan(z))\,e^{\frac{2y}{\beta}}\,dy^2
			\\ & \quad
				+k_3(\tan(z)-\tan(x))
					(\tan(z)-\tan(y))\,e^{\frac{2z}{\beta}}\,dz^2
		\end{align*}
		with the essential projective vector field
		$v=\partial_x+\partial_y+\partial_z$.
		\item\label{item:main.111.4}
		The metric, for $k_i\ne0$,
		\begin{align*}
			g &= k_1(\tan(x)-\tan(y))(\tan(x)-\tan(z))\,dx^2
				+k_2(\tan(y)-\tan(x))(\tan(y)-\tan(z))\,dy^2
			\\ & \quad
				+k_3(\tan(z)-\tan(x))(\tan(z)-\tan(y))\,dz^2
		\end{align*}
		with the essential projective vector field
		$v=\partial_x+\partial_y+\partial_z$.
		\item\label{item:main.111.5}
		The metric, for $k_i\ne0$,
		\begin{align*}
			g &= k_1(\tanh(x)-\tanh(y))(\tanh(x)-\tanh(z))\,dx^2
			\\ & \quad
				+k_2(\tanh(y)-\tanh(x))(\tanh(y)-\tanh(z))\,dy^2
			\\ & \quad
				+k_3(\tanh(z)-\tanh(x))(\tanh(z)-\tanh(y))\,dz^2
		\end{align*}
		with the essential projective vector field
		$v=\partial_x+\partial_y+\partial_z$.
		\item\label{item:main.21.1}
		The metric
		\[
			g = \frac{\beta}{z^2}\,(h+dz^2)
		\]
		with the essential projective vector field $v = \frac1z\,\partial_z$,
		where $h=h_{11}dx^2+2h_{12}dxdy+h_{22}dy^2$ does not admit any
		homothetic vector field.
		\item\label{item:main.21.2}
		The metric
		\[
			g = \beta\,(1+\tan^2(z))(h+dz^2)
		\]
		with the essential projective vector field $v = \tan(z)\,\partial_z$,
		where $h=h_{11}dx^2+2h_{12}dxdy+h_{22}dy^2$ does not admit any Killing
		vector field.
		\item\label{item:main.21.3}
		The metric
		\[
			g = \beta\,(1-\tanh^2(z))(h+dz^2)
		\]
		with the essential projective vector field $v = \tanh(z)\,\partial_z$,
		where $h=h_{11}dx^2+2h_{12}dxdy+h_{22}dy^2$ does not admit any Killing
		vector field.		
	\end{enumerate}	
\end{theorem}

\subsection{Structure and strategy of the paper}

The basic theory of metrisable projective connections, of projectively
equivalent metrics, and of projective vector fields (from the angle of Lie
theory of symmetries of differential equations) is reviewed in Section
\ref{Preliminaries}. An important object are certain (1,1)-tensors associated
to a pair of projectively equivalent metrics, called Benenti tensors. Their
eigenvalues are the functions $X_i$ and $\zeta$ that determine, respectively,
the metrics~\eqref{eqn:LC.111} and~\eqref{eqn:LC.21}.

As we explain later the projective class of these metrics is described by a 
2-dimensional linear space, 
cf.~Definition~\ref{def:Liouville.space.metr.space} and 
Proposition~\ref{prop:kiosak.matveev}.
This space is endowed with the
action of the projective symmetry algebra, studied in Section
\ref{sec:Lv.action}. Its matrix description defines a polynomial (called 
Solodovnikov's polynomial) which gives rise to differential equations
for the eigenvalues of the Benenti tensors, see Lemma
\ref{la:eigenvalues.L.roots.solodovnikov}. Together with the differential
equations arising from the projective symmetry action on the metric, see
Lemma \ref{la:proj.action.to.metrics}, this allows us to obtain the
explicit metrics by solving certain systems of ODEs. The details of the
procedure will be elaborated in Sections \ref{sec:111} and \ref{sec:21}.

In Section \ref{sec:LC.metrics} we formally introduce Levi-Civita metrics in
dimension $n$, which in the specific case $n=3$ lead to~\eqref{eqn:LC.111}
and~\eqref{eqn:LC.21}. The remaining two sections contain the main outcomes 
of the paper: In Section~\ref{sec:111}, we obtain local normal forms for 
\eqref{eqn:LC.111} (see Theorem~\ref{thm:1.1.1}) and in Section~\ref{sec:21} 
those for~\eqref{eqn:LC.21} (see Theorem \ref{thm:2.1}).

The proof of the main results obtained in this paper are based on a
combination of the methods from \cite{solodovnikov_1956} with the more recent
techniques from the classical Lie problem \cite{bryant_2008,matveev_2012},
from c-projective geometry \cite{bolsinov_2015cproj}, and from the
splitting-gluing theory of projectively equivalent metrics
\cite{bolsinov_2011,bolsinov_2015}.
The latter, in particular, allows us to reduce projective vector fields to
lower-dimensional homothetic vector fields, see
Propositions~\ref{prop:splitting.gluing.111} and
\ref{prop:proj.symmetries.descend.21type} as well as \ref{prop:gluing} for
details.
\medskip

Note the following notation that we apply from now on:
We use a comma to denote usual derivatives, e.g.~``$,a$'' denotes
differentiation w.r.t.~$a$-th coordinate direction. The comma will be omitted
when derivatives are interpreted as coordinates on certain jet spaces.
Unless otherwise clarified, Einstein's summation convention applies to
repeated indices.
We omit the comma and use a simple subscript to refer to coordinates on the 
jet space, see Section~\ref{sec:proj.conn.3d} for more details.

\section{Preliminaries}\label{Preliminaries}

\subsection{Metrisability of projective connections}

Let us now consider an $n$-dimensional metric $g$, given in terms of an 
explicit system of coordinates
\[
	(x^1,x^2,\dots,x^{n}) = (x,y^2,\dots,y^{n})\,.
\]
It gives rise, via its Levi-Civita connection, to a system of second order
ODEs
\begin{equation}\label{eqn:proj.conn.gen.multidim}
	y^k_{xx}=-\Gamma_{11}^k -\sum_{i=2}^n (2\Gamma_{1i}^k-\delta^k_i\Gamma_{11}^1)  y^i_x -\sum_{i,j=2}^n (\Gamma_{ij}^k  - 2\delta^k_i \Gamma_{1j}^1 )
	y^i_x y^j_x + \sum_{i,j=2}^n\Gamma_{ij}^1 \, y^i_x y^j_x y^k_x\,,\quad k=2,\dots,n\,,
\end{equation}
where $y^k=y^k(x)$ and where
\begin{equation}\label{eqn:Chris}
	\Gamma^k_{ij}=\frac12 g^{kh}(g_{jh,i}+g_{ih,j}-g_{ij,h})
\end{equation}
are the Christoffel symbols of the Levi-Civita connection of $g$.
System \eqref{eqn:proj.conn.gen.multidim} is called \emph{the projective
connection associated to $g$}. The name is justified by the fact that, for a
solution $y^k(x)$ to \eqref{eqn:proj.conn.gen.multidim}, the curve
$(x,y^2(x),\dots,y^n(x))$ is a geodesic of $g$ up to reparametrisation. In
fact, System \eqref{eqn:proj.conn.gen.multidim} can be achieved by
eliminating the external parameter from the classical geodesic
equation~\eqref{eqn:geodesic.equation}.

\begin{remark}\label{rem.proj.symm.as.point}
	In the context of the theory of symmetries of (systems of) differential
	equations, local diffeomorphisms $(x^1,\dots,x^n)\to
	\big(\tilde{x}^1(x^1,\dots,x^n),\dots,\tilde{x}^n(x^1,\dots,x^n)\big)$
	preserving \eqref{eqn:proj.conn.gen.multidim} are called \emph{finite
	point symmetries}.
	These send solutions of~\eqref{eqn:proj.conn.gen.multidim} to solutions
	and are projective transformations of~$g$, because solutions to
	\eqref{eqn:proj.conn.gen.multidim} are unparametrised geodesics of $g$ by
	construction. Infinitesimal point symmetries of
	\eqref{eqn:proj.conn.gen.multidim} are projective vector fields of $g$.
\end{remark}

\begin{example}
For $n=2$, in particular, and with $(x,y^2)=(x,y)$, System
\eqref{eqn:proj.conn.gen.multidim} reduces to the classical $2$-dimensional
projective connection associated to a $2$-dimensional metric
\cite{beltrami_1865}
\begin{equation}\label{eqn:projective.connection.2.dim}
	y_{xx} = -\Gamma^2_{11} +(\Gamma^1_{11}-2\Gamma^2_{12})\,y_x -(\Gamma^2_{22}-2\Gamma^1_{12})\,y_x^2 +\Gamma^1_{22}\,y_x^3\,.
\end{equation}
This equation has been extensively studied in \cite{bryant_2008}.
\end{example}

A general ($n$-dimensional) projective connection has the following form:
\begin{equation}\label{eqn:proj.conn.gen.multidim.2}
	y^k_{xx}=f_{11}^k + \sum_{i=2}^n f_{1m}^k  y^i_x + \sum_{i,j=2}^n f_{ij}^k
	y^i_x y^j_x + \sum_{i,j=2}^n f_{ij} \, y^i_x y^j_x y^k_x\,,\quad k=2,\dots,n\,\,,
\end{equation}
where w.l.o.g.\ we can suppose $f^k_{ij}$ and $f^1_{ij}:=f_{ij}$ symmetric in 
the lower indices.
A natural question is whether or not the projective connection
\eqref{eqn:proj.conn.gen.multidim.2} is metrisable, i.e.\ if there exists an
$n$-dimensional metric~$g$ such that \eqref{eqn:proj.conn.gen.multidim} is
equal to \eqref{eqn:proj.conn.gen.multidim.2}. This is equivalent to the
existence of a solution to the system
\begin{equation}\label{eqn:non.lin.syst}
	-\Gamma_{11}^k=f_{11}^k \,,\,\,\, - (2\Gamma_{1i}^k-\delta^k_i\Gamma_{11}^1)=  f_{1i}^k  \,,\,\,\,  -(\Gamma^k_{ij}-\delta^k_i\Gamma^1_{1j}-\delta^k_j\Gamma^1_{1i}) = f_{ij}^k\,,\,\,\,
	\Gamma_{ij}^1 =f_{ij}
\end{equation}
where $\Gamma^k_{ij}$ is given by \eqref{eqn:Chris}.
\begin{definition}
	The projective connection \eqref{eqn:proj.conn.gen.multidim.2} is
	\emph{metrisable} if there exists a Levi-Civita connection $\nabla$ such
	that~\eqref{eqn:non.lin.syst} is satisfied.
\end{definition}
\noindent If in \eqref{eqn:non.lin.syst} we perform the substitution
\begin{equation}\label{eqn:sigma.in.terms.of.g}
	\sigma^{ij}=|\det(g)|^{\frac{1}{n+1}}g^{ij}\in S^2(M)\otimes (\Lambda^n(M))^{\frac{2}{n+1}}\,,
\end{equation}
where $g^{ij}$ are the entries of the inverse metric of $g_{ij}$,
we obtain a \emph{linear} system in the unknowns $\sigma^{ij}$ (see
\cite{eastwood_2008}). The inverse transformation is
\begin{equation}\label{eqn:g.in.terms.of.sigma}
	g^{ij}=|\det(\sigma)|\sigma^{ij}\,.
\end{equation}
More precisely we have the following proposition.
\begin{proposition}[\cite{eastwood_2008}]\label{prop:lin.system}
	A metric $g$ on an $n$-dimensional manifold lies in the projective class
	of a given connection~$\nabla$ if and only if $\sigma^{ij}$ defined by
	\eqref{eqn:sigma.in.terms.of.g} is a solution of the linear system
	\begin{equation}\label{eqn:linear.system}
		\nabla_a\sigma^{bc}-\frac{1}{n+1}(\delta^c_a\nabla_i\sigma^{ib} + \delta^b_a\nabla_i\sigma^{ic} )=0\quad
\end{equation}
with
\begin{equation*}
\nabla_a\sigma^{bc} = \sigma^{bc}_{,a} + \Gamma^b_{ad}\sigma^{dc} + \Gamma^c_{ad}\sigma^{db} - \frac{2}{n+1}\Gamma^d_{da}\sigma^{bc}\,,
	\end{equation*}
	where $\Gamma^k_{ij}$ are the Christoffel symbols of the Levi-Civita
	connection of $g$.
\end{proposition}
The metrisability of projective connections can also be given in terms of
differential invariants, see~\cite{bryant_2009} and~\cite{dunajski_2016} for
2-dimensional and 3-dimensional projective structures, respectively.
A characterisation of metrisability in terms of pseudo-holomorphic curves can
be found in~\cite{mettler_2021}.

We proceed with two definitions that follow the conventions in
\cite{manno_2019}.
\begin{definition}\label{def:Liouville.space.metr.space}
	The linear space of solutions of \eqref{eqn:linear.system} is denoted by
	\[
		\solSp=\solSp(g):=\left\{\sigma\in  S^2(M)\otimes
		(\Lambda^n(M))^{\frac{2}{n+1}}\,\,|\,\, \sigma=\sigma^{ij} \text{ is
		a 	solution to } \eqref{eqn:linear.system}\right\}\,.
	\]
We refer to $\solSp$ as the \emph{metrisation space}. Its dimension
is called the \emph{degree of mobility} of $g$.
If $\sigma\in\solSp$ is represented by a
matrix of rank $r$, we say that $\sigma$ is a \emph{rank-$r$ solution}. If
$r<n$, the solution is also called \emph{degenerate} or a \emph{lower-rank
solution}. If $r=n$, the solution is called \emph{full-rank solution}.
\end{definition}

\begin{proposition}\label{prop:g.inverse}
Let $g$ and $\bar g$ be two $n$-dimensional projectively equivalent metrics
on the same $n$-dimensional manifold. Then, if defined for $t_1,t_2\in\R$,
\begin{equation}\label{eqn:g.inverse}
\frac{
	\big(
		t_1\,\det(g)^{\frac{1}{n+1}}g^{-1}
		+t_2\,\det(\bar g)^{\frac{1}{n+1}}\bar g^{-1}
	\big)^{-1}
	}{
	\det\left(
			t_1\,\det(g)^{\frac{1}{n+1}}g^{-1}
			+t_2\,\det(\bar g)^{\frac{1}{n+1}}\bar g^{-1}
		\right)}
\end{equation}
is a  metric  that is projectively equivalent to $g$ and $\bar g$.
\end{proposition}
\begin{proof}
To the metrics $g,\bar g$ correspond, respectively, the tensors
$\sigma,\bar\sigma$ via formula \eqref{eqn:sigma.in.terms.of.g}.
Then the claim follows taking into account the linearity of System
\eqref{eqn:linear.system}.
\end{proof}

\subsection{Benenti tensors and projective equivalence}

\begin{definition}\label{def:benenti.tensor}
	Let $g,\bar{g}$ be two projectively equivalent metrics on the same 
	$n$-dimensional manifold~$M$. The $(1,1)$-tensor field
	\begin{equation}\label{eqn:benenti.tensor}
	L(g,\bar{g}) =
	\left|\frac{\det(\bar{g})}{\det(g)}\right|^{\frac{1}{n+1}}\,\bar{g}^{-1}g
	\end{equation}
	on $M$ is called the \emph{Benenti tensor} associated with $(g,\bar{g})$.
	In view of Formula~\eqref{eqn:g.in.terms.of.sigma}, we also set
	\begin{equation}\label{eqn:benenti.tensor.from.sigma}
		L(\sigma,\bar\sigma)=\bar\sigma\sigma^{-1}.
	\end{equation}
\end{definition}

\begin{proposition}
\label{prop:Riem.Benenti}
Let $g,\bar{g}$ be two projectively equivalent metrics on the same
$n$-dimensional manifold~$M$. Then $L=L(g,\bar{g})$ is self-adjoint
w.r.t.~$g$. In particular, if $g$ is a Riemannian metric, then $L(g,\bar{g})$ 
is
diagonalisable.
\end{proposition}
\begin{proof}
It will be sufficient to prove that the matrix $(g L)_{ik}=g_{ij}L\indices{^j_k}$ is symmetric. Indeed, setting $f=\left|\frac{\det(\bar{g})}{\det(g)}\right|^{\frac{1}{n+1}}$ and denoting by $A^T$ the transpose of matrix $A$, we have that
$$
(gL)^T=(fg\bar{g}^{-1}g)^T=fg^T(\bar{g}^{1})^Tg^T=fg\bar{g}^{-1}g=gL\,.
$$
\end{proof}

\begin{remark}\label{rem:gLk.symmetric}
By induction, $gL^k$ is symmetric for any $k\in\mathbb{N}$. Indeed, supposing $gL^{k-1}$ symmetric, we have that
$$
	(gL^k)^T
	=(gL^{k-1}L)^T
	=L^TgL^{k-1}
	=(gL)^TL^{k-1}
	=(gL)L^{k-1}
	=gL^k\,.
$$
\end{remark}

\noindent If $g$ is of mixed signature, $L(g,\bar{g})$ is not necessarily
diagonalisable. For instance, consider
\[
	g=2(y^2+x)dxdy\text{ and }
	\bar{g}=-2\frac{y^2+x}{y^3}dxdy+\frac{(y^2+x)^2}{y^4}dy^2\,.
\]
These Lorentzian metrics are studied in \cite{matveev_2012}. 
The metrics $g$ and $\bar g$ are projectively equivalent as their projective 
connection is (see~\eqref{eqn:projective.connection.2.dim})
\[
	y_{xx}=\frac{1}{y^2+x}y_x-\frac{2y}{y^2+x} y_x^2\,,
\]
and $L(g,\bar{g})$ is not diagonalisable because
\[
	L(g,\bar{g})=
	-\left(
	\begin{array}{cc}
	y & y^2+x
	\\
	0 & y
	\end{array}
	\right)\,.
\]
\noindent We quote the following proposition, which follows from a more
general statement in the reference.
\begin{proposition}[\cite{bolsinov_2011}, Theorem 1]\label{prop:integrability}
Let $g$ and $\bar{g}$ be projectively equivalent  metrics. Let us assume that $L(g,\bar{g})$ is diagonalisable. Then the eigendistributions of $L(g,\bar{g})$ are integrable.
\end{proposition}

\subsection{Metrisability and the degree of mobility in dimension 3}

\begin{proposition}[\cite{kiosak_2010}, Corollary 3 p. 413]\label{prop:kiosak.matveev}
	The degree of mobility of a 3-dimensional metric $g$ of non-constant
	curvature is at most 2.
\end{proposition}

\begin{remark}\label{rem:degree.mob.1}
	In case the degree of mobility is 1, any metric $\bar{g}$ projectively 
	equivalent to $g$ is a constant multiple of $g$, i.e.~$\bar{g}=c\,g$ with 
	$c\ne0$.
	Thus, for such $g$, any projective vector field $w\ne0$ is homothetic. In 
	this case, in a sufficiently small neighborhood around a non-singular 
	point of $w$, we may thus choose coordinates $(x,y,z)$ such that 
	$w=\partial_x$, and this choice implies that the metric $g$ can locally 
	be written as in~\eqref{eqn:homothetic.normal.form}.
\end{remark}

The following lemma is useful for the understanding eigendistributions of
Benenti tensors $L(g,\bar{g})$ for 3-dimensional metrics.
\begin{lemma}\label{lemma:indep.mult.eigen}
	Let $g$ be an $n$-dimensional metric of non-constant curvature with
	degree of mobility $2$. Let~$\bar{g}$ be a metric projectively equivalent
	to $g$ such that $\bar{g}\neq c g$, $\forall\, c\in\mathbb{R}$.
	Then the number and multiplicity of distinct eigenvalues of the Benenti
	tensor $L=L(g,\bar{g})$ does not depend on the choice of $\bar{g}$.
\end{lemma}
\begin{proof}
	The sought functions are solutions of $\det(L-s\,\Id)=0$. Set
	$\sigma^{ij}=|\det(g)|^{\frac{1}{n+1}}g^{ij}$ and $\bar{\sigma}$
	analogously (cf. \eqref{eqn:sigma.in.terms.of.g}).
	We now replace the metric $\bar{g}$, using
	$\hat{g}=|\det(\hat\sigma)|^{-1}\hat{\sigma}^{-1}$ instead. Since we
	assume that the degree of mobility is equal to $2$, there exist two real
	constants $c_1,c_2$ such that
	\[
	\hat\sigma = c_1\sigma+c_2\bar{\sigma}\,,\quad \text{with}\,\,c_2\neq 0\,.
	\]
	As a consequence, we replace $L=L(g,\bar{g})$ by
	$\hat L=L(g,\hat{g})=\hat{\sigma}\sigma^{-1}=c_1 \Id+c_2L$.
	Therefore, a solution~$t$ of $\det(\hat L-t\,\Id)=0$ is in one-to-one
	correspondence with a solution of $\det(L-s\,\Id)=0$. Indeed,
	\[
	0 = \det(\hat L-t\,\Id)
	= c_2^n\,\det\left( L+\frac{1}{c_2}(c_1-t)\Id\right)\,,
	\]
	and thus we identify $s=\frac{c_1}{c_2}-\frac{t}{c_2}$. Note that the
	eigenvalues themselves change, but their number and multiplicity is
	preserved.
\end{proof}

\begin{corollary}
Let $g$ be a metric on a  $3$-dimensional manifold.
Let $\bar g$ and $\hat g$ be metrics on the same manifold and
projectively equivalent with $g$.
Assume $g$, $\bar g$ and $\hat g$ are pairwise non-proportional.
Then the number and multiplicity of eigenvalues of the Benenti tensor
$L(g,\bar g)$ corresponds with those of $L(g,\hat g)$.
\end{corollary}
\begin{proof}
This follows from Lemma~\ref{lemma:indep.mult.eigen} in combination with
Proposition~\ref{prop:kiosak.matveev}.
\end{proof}

\begin{lemma}\label{la:full.rank.basis}
	There is a metrisable $3$-dimensional projective connection such that
	$\dim\solSp=2$ if and only if there is a rank-3 basis
	$(\sigma,\bar\sigma)$ of $\solSp$.
\end{lemma}
\begin{proof}
	The backwards direction is trivial.
	Thus let us assume $\dim\solSp=2$ and the existence of a rank-3 solution
	$\sigma$ (otherwise the projective connection would not be metrisable).
	Then, since~\eqref{eqn:linear.system} is linear, all multiples $k\sigma$
	with $k\ne0$ are rank-3 solutions.
	Now, since the condition $\rank(\sigma)<3$ is a closed condition, the set
	$\solSp\setminus\{\sigma|\det\sigma=0\}$ is open.
	Thus,
there exists another rank-3 solution, say
	$\bar\sigma$, that is not proportional to $\sigma$.
	Since $\dim\solSp=2$ the pair $(\sigma,\bar\sigma)$ forms a rank-3 basis
	of $\solSp$.
\end{proof}

\subsection{Projective connections and projective vector fields in dimension
3}\label{sec:proj.conn.3d}

In coordinates $(x^1,x^2,x^3)=(x,y,z)$, in view of \eqref{eqn:proj.conn.gen.multidim.2}, a 3-dimensional projective connection is given by
\begin{equation}\label{eqn:3.dim.proj.conn}
\left\{
\begin{array}{rl}
y_{xx}
&=f^2_{11} + f^2_{12}y_x + f^2_{13}z_x + f^2_{22}y_x^2 + 2f^2_{23}y_x z_x +
f^2_{33}z_x^2 + f_{11}y_x^3 + 2f_{12}y_x^2 z_x + f_{22}y_x z_x^2
\\
&=:F^2(x,y,z,y_x,z_x)
\\
\\
z_{xx}
&=f^3_{11} + f^3_{12}y_x + f^3_{13}z_x + f^3_{22}y_x^2 + 2f^3_{23}y_x z_x +
f^3_{33}z_x^2 + f_{11}y_x^2  z_x + 2f_{12}y_x  z_x^2 + f_{22} z_x^3
\\
&=:F^3(x,y,z,y_x,z_x)
\end{array}
\right.
\end{equation}
Taking into account Remark \ref{rem.proj.symm.as.point}, below we write the system of differential equations that the components of a vector field
\begin{equation}\label{eqn:v}
v=v^1\partial_x+v^2\partial_y+v^3\partial_z\,,\quad v^i=v^i(x,y,z)
\end{equation}
have  to satisfy to be a projective vector field for the projective
connection \eqref{eqn:3.dim.proj.conn}. For this purpose, it is useful to recall some basic facts
concerning the theory of infinitesimal symmetries of a system of 
type~\eqref{eqn:3.dim.proj.conn}. More details can be found, for
instance, in \cite{olver2000applications}.
Consider the second jet space $J^2(1,2)$ with an independent variable $x$ and 
two dependent variables $y,z$ (since we are working locally, we may think of
$J^2(1,2)$ as $\mathbb{R}^7$ with coordinates  $(x,y,z,y_x,z_x,y_{xx},z_{xx})$). We define the total derivative operator restricted to system
 \eqref{eqn:3.dim.proj.conn}:
\begin{equation*}
D_x=\partial_x+y_x\partial_y+z_x\partial_z+F^2\partial_{y_x}+F^3\partial_{z_x}\,.
\end{equation*}
The prolongation of vector field \eqref{eqn:v}  to the second jet $J^2(1,2)$ is a vector field given by
\begin{multline}
v^{(2)}=v^1\partial_x+v^2\partial_y+v^3\partial_z+\left(D_x(v^2)-y_xD_x(v^1)\right)\partial_{y_x} + \left(D_x(v^3)-z_xD_x(v^1)\right)\partial_{z_x} +
\\
\left(D_x^2(v^2)-y_xD_x^2(v^1)-2y_{xx}D_x(v^1)\right)\partial_{y_{xx}} + \left(D_x^2(v^3)-z_xD_x^2(v^1)-2z_{xx}D_x(v^1)\right)\partial_{z_{xx}}
\end{multline}
where $D^2_x=D_x\circ D_x$. The above vector field is determined by the local
one-parametric group of transformations $\phi^{(2)}_t$ given by the
prolongation to the second jet space $J^2(1,2)$ of the local flow $\phi_t$ 
of~\eqref{eqn:v} (\cite[Sec. 2.3]{olver2000applications}).

The vector field \eqref{eqn:v} is an infinitesimal symmetry of system  
\eqref{eqn:3.dim.proj.conn} (i.e.\ in other words, a projective vector field 
of the projective connection  \eqref{eqn:3.dim.proj.conn}, see again Remark 
\ref{rem.proj.symm.as.point}) if and only if
\begin{equation}\label{eqn:cond.proj}
\left.
\begin{array}{l}
v^{(2)}(y_{xx}-F^2)=0
\\
v^{(2)}(z_{xx}-F^3)=0
\end{array}
\right\}
\,\text{when restricted to\,\,} \{y_{xx}=F^2\,,z_{xx}=F^3\}\,.
\end{equation}
A straightforward computation shows that conditions \eqref{eqn:cond.proj} give
a second order system of 18 partial differential equations (PDEs)  in the unknown functions $(v^1,v^2,v^3)$.

\section{The action of the projective algebra\texorpdfstring{ on the
metrisation space $\solSp$}{}.}
\label{sec:Lv.action}

The action $\Phi$ of a projective vector field $v\in\projalg$ on $\solSp$ is
\begin{equation}\label{eq:proj.action}
  \Phi:\projalg\times\solSp\to\solSp\,,\quad
  \Phi(v,\sigma)=\lie_v\sigma\,,
\end{equation}
where $\lie_v$ is the Lie derivative along $v$.
\begin{remark}
	Note that $\solSp$ is invariant under $\Phi$ 
	\cite{matveev_2012,eastwood_2008}.
	This can be proven in an elementary way by choosing local coordinates
	$(x,y_2,\dots,y_n)$ such that $v=\partial_x$. Then in
	\eqref{eqn:proj.conn.gen.multidim} the
	coefficients depend on $y_2,\dots,y_n$ only and the Lie derivative acts
	by usual derivatives. Therefore $\lie_v\sigma\in\solSp$
	\cite{eastwood_2008}.
\end{remark}

\noindent Since $\solSp\sim\mathbb{R}^2$, action \eqref{eq:proj.action}, for
a specific $v\in\projalg$, is encoded in a $2\times2$ matrix $A$
\cite{matveev_2012}:
\begin{equation}\label{eqn:abcd}
	\lie_v\begin{pmatrix} \sigma \\ \bar\sigma \end{pmatrix}
	:= \begin{pmatrix} \lie_v\sigma \\ \lie_v\bar\sigma \end{pmatrix}
	=A\begin{pmatrix} \sigma \\ \bar\sigma \end{pmatrix}\,.
\end{equation}
where
\begin{equation}\label{eq:A}
	A=
	\begin{pmatrix}
			a & b \\ c & d
		\end{pmatrix}\,,\quad a,b,c,d\in\mathbb{R}\,.
\end{equation}
The elements $\sigma,\bar\sigma\in\solSp$ can correspond to metrics
($\sigma,\bar\sigma$ are full-rank solutions), or they can be of lower rank.

\noindent Concerning~\eqref{eqn:abcd} we also remark the following:
	
\begin{itemize}
\item	For a tensor section $\sigma$
	of (projective) \emph{weight} $w$ (compare~\cite{eastwood_2008}), the Lie
	derivative is related to the covariant
	derivative by
	\[
		\lie_v\sigma^{ij} = v^k\nabla_k\sigma^{ij}
		-\nabla_kv^i\,\sigma^{kj}
		-\nabla_kv^j\,\sigma^{ik}
		+w\,\nabla_kv^k\,\sigma^{ij}\,.
	\]
	
	\item
	A non-degenerate, weighted tensor section $\sigma\in\solSp$ has weight
	$-2$.
	Let us recall that $g$ and $\sigma$ are related
	by~\eqref{eqn:sigma.in.terms.of.g} and \eqref{eqn:g.in.terms.of.sigma}. 
	Then their Lie derivatives are related by
	\begin{equation}\label{eqn:Lv.sigma}
		\lie_v\sigma^{ij}=|\det(g)|^{\frac{1}{n+1}}\left(  \frac{1}{n+1} 
		\trace(g^{-1}\lie_vg)g^{ij} + \lie_vg^{ij} \right)
	\end{equation}
	and
	\begin{equation}\label{eqn:Lvg.w.r.t.sigma}
		\lie_vg^{ij}=|\det(\sigma)|\big(\lie_v\sigma^{ij} +
		\trace(\sigma_{ac}\lie_v\sigma^{cb})\sigma^{ij}\big)
	\end{equation}
	
	\item
	In the present paper we follow the notation in~\cite{matveev_2012}.
	The matrix $A$ in \eqref{eqn:abcd} therefore is not the matrix
	representing, in the basis $(\sigma,\bar\sigma)$, the Lie derivative by
	the usual conventions, but its
	transpose.
\end{itemize}

The following lemma is based on similar statements in \cite{manno_2018} and
\cite{solodovnikov_1956}.

\begin{lemma}\label{la:normalising.Lw}
Let $g$ be a metric of degree of mobility 2 admitting a projective vector
field $v$.
Let $\sigma$ be as in~\eqref{eqn:sigma.in.terms.of.g}.
Then there exist a constant $k\ne0$ and a solution $\bar\sigma\ne0$ of
\eqref{eqn:linear.system} such that
\[
	\lie_{kv}\sigma = a\sigma+b\bar\sigma.
\]
where $\bar\sigma$ and $\sigma$ are not proportional and where
\begin{enumerate}
	\item $a=b=0$ if $v$ is a Killing vector field for $g$.
	\item $b=0,a=1$ if $v$ is homothetic for $g$, and not a Killing vector
	field for $g$.
	\item $a=0,b=1$ if $v$ is an essential projective vector field for $g$.
\end{enumerate}
\end{lemma}

\noindent The lemma normalises the constants $a,b$ in~\eqref{eq:A}.
Note however that in the lemma, for essential $v$, the solution $\bar\sigma$
is not necessarily of full rank.
\begin{proof}
If $v$ is Killing, then $\lie_vg=0$ and thus $\lie_v\sigma=0$. We set $k=1$.
Let $\bar g$ be projectively equivalent to $g$, but not proportional to $g$.
Then we obtain a solution $\bar\sigma$ of~\eqref{eqn:linear.system} from
$\bar g$ using~\eqref{eqn:sigma.in.terms.of.g}. For any such~$\bar\sigma$, we
find $a=b=0$.
If $v$ is properly homothetic, we have $\lie_v\sigma=\mu\sigma$ for some
constant $\mu\ne0$.
We decree $k=\mu^{-1}$, and obtain $\lie_{kv}\sigma=\sigma$. We conclude
$a=1,b=0$ analogously to the previous case.
Finally, if $v$ is an essential projective vector field, we let
$\bar\sigma:=\lie_v\sigma$ and $k=1$. Then $a=0,b=1$.
\end{proof}

We have the following lemma, found in~\cite{bolsinov_2011}, see
also~\cite{bolsinov_2015cproj}. It generalises equations
in~\cite{solodovnikov_1956}.
\begin{lemma}[\cite{bolsinov_2011}]
	Let $L=L(\sigma,\bar\sigma)$ with $\sigma$ of full rank. Then,
	for a projective vector field $v$,
	\begin{equation}\label{eqn:LvL.solpol}
		\lie_vL = -bL^2+(d-a)L+c\Id
	\end{equation}
\end{lemma}
\begin{proof}
	This follows from a direct computation, using the formula
	$\lie_v(\sigma^{-1})=-\sigma^{-1}(\lie_v\sigma)\sigma^{-1}$ that holds if $\sigma$ is not of lower rank. Indeed, on account of \eqref{eqn:abcd},
	\begin{align*}
		\lie_vL
		&= \lie_v(\bar\sigma\sigma^{-1})
		= \lie_v(\bar\sigma)\sigma^{-1} + \bar\sigma\lie_v\sigma^{-1}
		= \lie_v(\bar\sigma)\sigma^{-1}
		- \bar\sigma\sigma^{-1}(\lie_v\sigma)\sigma^{-1}
		\\
		&= (c\Id+dL) - L(a\Id+bL)
		= -bL^2+(d-a)L+c\Id\,.
	\end{align*}
\end{proof}

\begin{lemma}\label{lemma:Lwg1L.symm}
Let $L=L(g,\bar{g})$ with $g$ and $\bar{g}$ projectively equivalent. Let $v$
be a projective vector field of~$g$. Then $(\lie_vg)L$ is symmetric. In
particular $(\lie_vg)\big(L(u),w\big)=(\lie_vg)\big(u,L(w)\big)$.
\end{lemma}
\begin{proof}
We have that
$$
(\lie_vg)L=\lie_v(gL) - g(\lie_vL)\,.
$$
Since $gL$ is symmetric (see the proof of Proposition \ref{prop:Riem.Benenti}), also $\lie_v(gL)$ is symmetric. So, in order to prove the statement of the lemma, in view of the above formula, it would be enough to prove that $g(\lie_vL)$ is symmetric. Since $\lie_vL$ is given by \eqref{eqn:LvL.solpol}, this follows by Remark \ref{rem:gLk.symmetric}.
\end{proof}

\begin{definition}\label{def:solpol}
Let $v$ be a projective vector field. Let $\lie_v$ be described by a matrix $A$ as in \eqref{eqn:abcd}.	We call the formal polynomial
	\begin{equation}\label{eq:solodovnikov.polynomial}
		\solpol_A(t) = -bt^2+(d-a)t+c
	\end{equation}
	\emph{Solodovnikov's polynomial}.
\end{definition}

\begin{lemma}\label{la:eigenvalues.L.roots.solodovnikov}
	Let $g$ be a 3-dimensional metric with degree of mobility 2 admitting 
	projective vector field $v$.
	Let $\bar g$ be projectively equivalent, but non-proportional, to $g$. 
	Let $\lambda$ be an eigenvalue of $L=L(g,\bar g)$.
	Then
	\begin{equation}\label{eqn:Andreas.Solod.1}
		\solpol_A(\lambda) =v(\lambda)\,,
	\end{equation}
	with $\solpol_A$ given by \eqref{eq:solodovnikov.polynomial}. In
	particular, if $\lambda$ is constant, then it is a root of $\solpol_A$.
\end{lemma}
\begin{proof}
	Let $w$ be an eigenvector of $L$ with the eigenvalue $\lambda$, 
	i.e.~$Lw=\lambda w$.
	Because of the identity
	\[
	\lie_v(\lambda w) = \lie_v(Lw)
	= (\lie_vL)w + L\lie_vw
	= (\lie_vL)w + L[v,w]\,,
	\]
	and due to~\eqref{eqn:LvL.solpol}, we have
	\begin{equation}\label{eqn:Andreas.Solod.2}
		0
		= \solpol_A(L)w - (\lie_vL)w
		=\solpol_A(\lambda) w-v(\lambda) w +(L-\lambda\Id)\,[v,w]\,.
	\end{equation}
	This implies
	\begin{equation}\label{eqn:appoggio.gianni.1}
		\big(-\solpol_A(\lambda)+v(\lambda)\big) g(w,w) =
		g\big((L-\lambda\Id)\,[v,w] ,w\big)=0\,,
	\end{equation}
	where the second equality holds because of
	\[
	g\big((L-\lambda\Id)\,[v,w] ,w\big)
	\overset{(*)}{=}
	g\big([v,w],(L-\lambda\Id)w\big)
	=g\big([v,w] ,0\big)=0\,.
	\]
	The equality~$(*)$ above holds true as $gL-\lambda g$ is symmetric, see 
	the	proof of Proposition \ref{prop:Riem.Benenti}.
	
	If $g(w,w)\neq 0$, Equation~\eqref{eqn:appoggio.gianni.1} immediately 
	yields~\eqref{eqn:Andreas.Solod.1}.\footnote{%
		Note that this proves the claim, even in arbitrary dimension, for any 
		non-null eigenvector $w$, and in particular for Riemannian metrics.}
	Thus let $g(w,w)=0$.
	If $[v,w]$ is proportional to~$w$, i.e.~$[v,w]=\mu w$ for some function 
	$\mu$, then
	$(L-\lambda\Id)\,[v,w] =\mu\big( Lw-\lambda w\big)=0$ and, because
	of~\eqref{eqn:Andreas.Solod.2}, we obtain \eqref{eqn:Andreas.Solod.1}.
	It remains to prove the claim for $g(w,w)=0$ if $u:=[v,w]$ is not 
	proportional to $w$.
	By virtue of~\eqref{eqn:Andreas.Solod.2}, in this case we have
	\begin{equation}\label{eqn:Lu}
		Lu=v(\lambda) w+\lambda u -\solpol_A(\lambda) w\,.
	\end{equation}
	Hence Equation~\eqref{eqn:LvL.solpol} yields the first equality in
	\begin{equation}\label{eqn:S(L).LvL}
		\left( -bL^2+(d-a)L+c\Id \right) u = (\lie_vL)u
		= \lie_v(Lu)-L[v,u]\,,
	\end{equation}
	where the second equality follows from the Leibniz rule.
	In order to compute the left hand side of this identity, we 
	use~\eqref{eqn:Lu} to obtain
	\begin{equation}\label{eqn:L2u}
		L^2u=2v(\lambda)\lambda w + \lambda^2u-2\solpol_A(\lambda)\lambda w
	\end{equation}
	After resubstituting \eqref{eqn:Lu} and \eqref{eqn:L2u}
	into~\eqref{eqn:S(L).LvL}, using \eqref{eqn:Andreas.Solod.2}, we have
	\begin{equation}\label{eqn:nice}
		2(\solpol_A(\lambda)-v(\lambda))u + Bw = (\lambda\Id-L)[v,u]
	\end{equation}
	for some expression $B$.
	To finish the proof, we study~\eqref{eqn:nice} for the two possible cases:
	First, in the case when $g(u,w)\neq0$, Equation~\eqref{eqn:nice} together
	with Proposition~\ref{prop:Riem.Benenti} implies
	\begin{align*}
		2(\solpol_A(\lambda)-v(\lambda))g(u,w)
		=g\big([v,u],(\lambda\Id-L)w\big)=0\,.
	\end{align*}
	We obtain~\eqref{eqn:Andreas.Solod.1}.
	Second, in the case when $g(u,w)= 0$, we have
	\begin{equation}\label{eqn:lievgww.zero.2}
		(\lie_vg)(u,w)+g([v,u],w)+g(u,u) = v(g(u,w)) = 0\,,
	\end{equation}
	along with
	\begin{equation}\label{eqn:lievgww.zero}
		(\lie_vg)(w,w) = v(g(w,w)) = 0\,,
	\end{equation}
	which follows already from~$g(w,w)=0$.
	Therefore, using~\eqref{eqn:nice} and then Lemma~\ref{lemma:Lwg1L.symm}, 
	we conclude
	\begin{equation}\label{eqn:serve}
		2(\solpol_A(\lambda)-v(\lambda))(\lie_vg)(u,w)=
		(\lie_vg)\big([v,u],(\lambda\Id-L)w\big)
		=0\,.
	\end{equation}
	Note that the conditions $g(u,w)=0=g(w,w)$ imply, together with the
	assumption that $u$ and $w$ are non-proportional, that
	\begin{equation}\label{eqn:guu.not.zero}
		g(u,u)\neq 0\,,
	\end{equation}
	as otherwise $g$ would be degenerate.
	Taking the scalar product w.r.t.~$g$ of~\eqref{eqn:nice} and $u$, we
	arrive at
	\begin{multline*}
		2(\solpol_A(\lambda)-v(\lambda))g(u,u)
		\overset{\text{Prop. \ref{prop:Riem.Benenti}}}{=}
		g\big([v,u], (\lambda\Id-L)u\big)
		\overset{\eqref{eqn:Lu}}{=}
		g\big([v,u], \solpol_A(\lambda) w-v(\lambda)w \big)
		\\
		\overset{\eqref{eqn:lievgww.zero.2}}{=}
		-(\solpol_A(\lambda) -v(\lambda))\big((\lie_vg)(u,w)+g(u,u)\big)
		\overset{\eqref{eqn:serve}}{=}
		-(\solpol_A(\lambda) -v(\lambda))g(u,u)
	\end{multline*}
	implying \eqref{eqn:Andreas.Solod.1} in view of 
	\eqref{eqn:guu.not.zero}.
	This concludes the proof.
\end{proof}

In view of \eqref{eqn:Andreas.Solod.1} and \eqref{eqn:Andreas.Solod.2}, we obtain the following corollary.
\begin{corollary}
Let the hypotheses be as in Lemma \ref{la:eigenvalues.L.roots.solodovnikov}. Then
 $[v,w]$ is an eigenvector of $L$ with eigenvalue~$\lambda$.
\end{corollary}

\begin{lemma}\label{la:proj.action.to.metrics}
	Let $g$ and $\bar g$ be a pair of projectively equivalent,
	non-proportional $n$-dimensional metrics with non-constant curvature and
	degree of mobility 2.
	Let $g$ (and $\bar g$) admit the projective vector field~$v$.
	Then
	\begin{subequations}\label{eq:Lvgg}
		\begin{align}
			\mathcal{L}_vg_{ij}
			&= -(n+1)a\,g_{ij}
			- b \left( \tr(L)\,g_{ij}
			+ g_{ik}L\indices{^k_j} \right)
			\label{eq:Lvg1}
			\\
			\mathcal{L}_v\bar g_{ij}
			&=- (n+1)d\,\bar g_{ij}
			- c \left( \tr(\bar{L})\,\bar g_{ij}
			+ \bar g_{ik}\bar{L}\indices{^k_j} \right)
			\label{eq:Lvg2}
		\end{align}
	\end{subequations}
	where $L=L(g,\bar g)$ is the Benenti tensor of $(g,\bar g)$.
	By $\bar{L}$ we denote the inverse of $L$, $\bar{L}=L^{-1}$.
\end{lemma}
\begin{proof}
	Recall the formula $0=\lie_v(g^{-1}g)=\lie_v(g^{-1}) g+ g^{-1}\lie_vg$,
	whereby $\lie_vg=-g\lie_v(g^{-1})g$.
	By a direct computation using formula~\eqref{eqn:sigma.in.terms.of.g}, 
	and the identities~\eqref{eqn:abcd} and~\eqref{eqn:Lvg.w.r.t.sigma}, it 
	then follows that
	\begin{equation*}
		\lie_vg
		= -\det(\sigma)\left[
		g(a\sigma+b\bar\sigma)g
		+ \tr\left(\sigma^{-1}(a\sigma+b\bar\sigma)\right)g\sigma g 
		\right]
		= -\big( (n+1)ag +bgL+b\tr(L)g \big)\,.
	\end{equation*}
	This proves Equation~\eqref{eq:Lvg1}.
	Equation~\eqref{eq:Lvg2} is obtained analogously.
\end{proof}

\section{Levi-Civita metrics}\label{sec:LC.metrics}

\begin{definition}\label{def:Levi-Civita}
A metric $g$ is said to be a \emph{Levi-Civita metric} if there exits a
metric $\bar{g}$ on the same manifold $M$ such that
\begin{itemize}
	\item $g$ and $\bar{g}$ are projectively equivalent metrics,
	\item $\bar{g}$ is not proportional to $g$, and
	\item $L(g,\bar{g})$ is diagonalisable.
\end{itemize}
\end{definition}

\begin{remark}\label{rmk:dom.is.two.or.cc}
	Let $g$ be a 3-dimensional Levi-Civita metric. Then it is either of constant curvature
	or its degree of mobility is exactly 2.
	The reason is that by Proposition~\ref{prop:kiosak.matveev} the metric is
	either of constant curvature or the degree of mobility is \emph{at most}
	2. But if the degree of mobility is equal to one, then any metric~$\bar
	g$ projectively equivalent to $g$ is proportional to $g$, contradicting
	the fact that $g$ is a Levi-Civita metric.
\end{remark}

\begin{proposition}\label{prop:always.diago}
Let us assume that the degree of mobility of a metric $g$ is 2. Then $g$ is
a Levi-Civita metric if and only if
$L(g,\bar{g})$ is
diagonalisable for all metrics $\bar{g}$ on the same manifold which are
projectively equivalent to $g$.
\end{proposition}
\begin{proof}
The ``$\Leftarrow$'' implication follows immediately from
Definition~\ref{def:Levi-Civita}.
For the ``$\Rightarrow$'' implication we assume that the degree of mobility
of $g$ is 2.
By definition of Levi-Civita metrics, there exists $\bar g$ non proportional to $g$ and  such that $L=L(g,\bar g)$ is
diagonalisable.
Let $\sigma$ and $\bar\sigma$ be the corresponding solutions for
$g,\bar g$, respectively, according to \eqref{eqn:sigma.in.terms.of.g}. Then
$\sigma$ and $\bar\sigma$ are linearly independent.
Since the degree of mobility is $2$, we obtain that any other solution
$\hat\sigma$ of \eqref{eqn:linear.system} is a linear combination of $\sigma$
and~$\bar\sigma$, with $s\ne0\ne t$,
\[
  \hat\sigma = s\sigma+t\bar\sigma\,.
\]
Let $\hat L=L(\sigma,\hat\sigma)$. The characteristic polynomial of $\hat L$
is
\[
   \det(\hat L-\lambda\Id)
  = \det(\hat\sigma\sigma^{-1}-\lambda\Id)
  = \det((s\sigma+t\bar\sigma)\sigma^{-1}-\lambda\Id)
  = \det(tL-(\lambda-s)\Id)
\]
Therefore, any eigenvalue $\lambda$ of $\hat L$ corresponds with an
eigenvalue $\lambda'=\frac{\lambda-s}{t}$ of $L$, and vice versa. This confirms that the algebraic multiplicity is the same.
We now verify that also the geometric multiplicities coincide. For a specific
eigenvalue $\lambda$ of $\hat L$, the eigenspace is (note $t\ne0$)
\[
  \ker(\hat L-\lambda \Id)=\ker(tL+s\Id-\lambda\Id)
  = \ker(tL-t\lambda'\Id)
  = \ker(L-\lambda'\Id)\,.
\]
It follows that $\hat L$ is diagonalisable, since $L$ is diagonalisable.
\end{proof}

In order to have a local description of Levi-Civita metrics of non-constant curvature and with degree of mobility equal to 2, we can proceed taking into account the following considerations:
\begin{enumerate}
\item[1)] Since such metrics have degree of mobility equal to 2, in view of
Proposition \ref{prop:g.inverse}, it will be enough to have a local
description of a \emph{pair} of metrics $g$ and $\bar g$ in the considered
projective class in a common system of coordinates;
\item[2)] This is facilitated by the  diagonalisability of $L(g,\bar g)$ (see
Proposition \ref{prop:always.diago}) and by the integrability of its
eigendistributions (see Proposition \ref{prop:integrability}).
\end{enumerate}
Concerning point $2)$, let us assume to have $m$ integrable
eigendistributions $\mathcal{D}_1,\dots,\mathcal{D}_m$ of $L(g,\bar g)$, with
$\dim(\mathcal{D}_i)=k_i$. We can then choose
a system of coordinates $(x^1,\dots,x^n)$ on $M$
\begin{equation*}
(x^1,\dots,x^n)=\big(x^1_1,\dots,x^{k_1}_1,x^1_2,\dots,x^{k_2}_2,\dots,x^1_m,\dots,x^{k_m}_m\big)\,,\quad\sum_{i=1}^m k_i=n\,,
\end{equation*}
such that $\mathcal{D}_i=\langle \partial_{x_i^1}\,,\dots\,,\partial_{x_i^{k_i}} \rangle$. For cosmetic reasons, w.l.o.g., we assume that the coordinates corresponding to $1$-dimensional eigendistributions are the first ones, i.e.\ the above system of coordinates is rearranged as follows:
\begin{equation}\label{eqn:local.coordinates.ordered}
(x^1,\dots,x^n)=\big(x^1_1,x^1_2,\dots, x^1_r,x^{1}_{r+1},\dots,x^{k_{r+1}}_{r+1},\dots,x^1_m,\dots,x^{k_m}_m\big)\,,\quad\sum_{i=r+1}^m k_i=n-r\,,\,\,k_i\geq 2\,.
\end{equation}

\begin{proposition}[\cite{bolsinov_2011,bolsinov_2015,matveev_2012rel}]
	\label{prop:Levi-Civita.normal.forms}
Let $g$ be a Levi-Civita metric. Then there exists a metric $\bar g$ that is
projectively equivalent and non-proportional to $g$, and (almost everywhere,
in a sufficiently small neighborhood) local
coordinates~\eqref{eqn:local.coordinates.ordered} such that $g,\bar g$ assume
the form
\begin{subequations}\label{eqn:LC.general}
	\begin{align}
		g &= \sum_{i=1}^r P_i (dx^1_i)^2
		+\sum_{i=r+1}^m \left[ P_i
		\sum_{\alpha_i,\beta_i}^{k_i}
		(h_i(x_i))_{\alpha_i,\beta_i} dx_i^{\alpha_i}dx_i^{\beta_i}
		\right]
		\label{eqn:g1.LC.general}
		\\
		\bar g &= \sum_{i=1}^r P_i\rho_i (dx^1_i)^2
		+\sum_{i=r+1}^m \left[ P_i\rho_i
		\sum_{\alpha_i,\beta_i}^{k_i}
		(h_i(x_i))_{\alpha_i,\beta_i} dx_i^{\alpha_i}dx_i^{\beta_i}
		\right]
		\label{eqn:g2.LC.general}
	\end{align}
\end{subequations}
where $(x_i)$ stands for the set of coordinates $(x_i^1,\dots,x_i^{k_i})$ and where
\[
	P_i = \pm\prod (X_i-X_j)\qquad\text{and}\qquad
	\rho_i = \frac1{X_i\prod_\alpha X^{k_\alpha}_\alpha}\,,
\]
with $X_i$ denoting the eigenvalue of $L(g,\bar g)$ for the eigendistribution
$\mathcal{D}_i$. Here, the $k_i$ are numbers larger or equal to 2.
\end{proposition}
\begin{remark}\label{rmk:higher.multiplicities}
	The metrics in~\eqref{eqn:LC.general} are local. According to
	\cite{bolsinov_2011} they can be achieved after a coordinate change, in a
	sufficiently small neighborhood of a point where $\mathrm{image}(X_i)\cap\mathrm{image}(X_j)=\emptyset$, $\forall\,i\neq j$. We remark that one may assume $X_i>X_{i+1}$ for
	$1\leq i\leq r$ and, respectively, for $r+1\leq i\leq m$.
	
	Due to \cite[Theorem~3]{bolsinov_2015}, the $X_i$ are constants for
	$i\geq r+1$, i.e.\ for the building blocks with geometric multiplicity
	$k_i\geq2$, see also \cite{Weyl_1921,levi-civita_1896} and \cite[Lemma
	6]{kiosak_2009}.
	For $i\leq r$, in the coordinates of \eqref{eqn:LC.general} the
	$X_i=X_i(x_i^1)$ are univariate functions.
\end{remark}

\subsection{Levi-Civita metrics of dimension 3}\label{sec:Cases.111.21}

 The current paper concerns itself with Levi-Civita metrics in dimension~$3$. Due to Proposition~\ref{prop:Levi-Civita.normal.forms} we therefore have to study two distinct cases:

\medskip

\noindent
	 \boxed{\textbf{[1-1-1] metrics}} In this case $r=m=3$, and 
	 \eqref{eqn:LC.general} gives
			\begin{subequations}\label{eqn:111}
				\begin{align}
					g &= \pm(X_1-X_2)(X_1-X_3)\,(dx^1)^2
							\pm(X_2-X_1)(X_2-X_3)\,(dx^2)^2
							\pm(X_3-X_1)(X_3-X_2)\,(dx^3)^2
					\label{eqn:111.1}
					\\
					\bar g &=
					\pm\frac{(X_1-X_2)(X_1-X_3)}{X_1^2X_2X_3}\,(dx^1)^2
						\pm\frac{(X_2-X_1)(X_2-X_3)}{X_1X_2^2X_3}\,(dx^2)^2
						\pm\frac{(X_3-X_1)(X_3-X_2)}{X_1X_2X_3^2}\,(dx^3)^2
					\label{eqn:111.2}
				\end{align}
			\end{subequations}
			and without loss of generality it may be assumed that
			$X_3(x^3)>X_2(x^2)>X_1(x^1)$. Indeed, the eigenvalues of
			$L(g,\bar g)$ are not
			allowed to coincide as otherwise already the following case would
			occur.

\medskip

\noindent
	\boxed{\textbf{[2-1] metrics}} In this case $r=1$, $m=2$ and 
	\eqref{eqn:LC.general} gives
			\begin{subequations}\label{eqn:21.before}
			\begin{align}\label{eqn:21.before.1}
				g &= \pm(X_1-X_2)\,\left[
							(dx^1)^2
							\mp h
						\right]
				\\
				\bar g &= \pm\frac{X_1-X_2}{X_1X_2^2}\,\left[
							\frac{(dx^1)^2}{X_1}
							\mp \frac{h}{X_2}
						\right]
			\end{align}
			\end{subequations}
			where $h=\sum_{i,j=2}^3 h_{ij}dx^idx^j$ is any 2-dimensional metric, $X_1=X_1(x^1)$ and $X_2$ is a constant
			(cf. Remark \ref{rmk:higher.multiplicities}).
			For cosmetic reasons we are going to rename the coordinates
			such that the manifold underlying $h$ has coordinates $x$ and $y$
			and such that the manifold underlying one-dimensional distribution
			has the coordinate $z$. We thus have
			\begin{subequations}\label{eqn:21}
			\begin{align}
				g &= \zeta(z)\,(h \pm dz^2)
				\label{eqn:21.1}
				\\
				\bar g &= \frac{\zeta(z)}{Z(z)\,\rho^2}\,\left(
							\frac{h}{\rho} \pm \frac{dz^2}{Z(z)}
						\right)
				\label{eqn:21.2}
			\end{align}
			\end{subequations}
			where $h=h_{11}dx^2+2h_{12}dxdy+h_{22}dy^2$, $h_{ij}=h_{ij}(x,y)$, and
			$$
				\zeta(z)=Z(z)-\rho\,,\quad\rho\in\R\,.
			$$
			The manifold $M$ underlying $g$ and $\bar g$ is a product of a
			2-dimensional manifold $M_2$ with coordinates $x,y$ and a
			1-dimensional manifold $M_1$ with coordinate $z$.

\medskip

\noindent
Note that the case $r=0$, $m=1$ cannot appear.
Indeed, if $g$ were such a Levi-Civita metric then there would exist a metric $\bar{g}$ projectively equivalent to $g$ but not proportional to it. The tensor \eqref{eqn:benenti.tensor} thus has one eigenvalue of multiplicity three, i.e.\ it is a multiple of the identity. Consequently $g$ and $\bar{g}$ must be already proportional, contradicting the hypothesis to be of Levi-Civita type.

\begin{example}\label{ex:LC.metrics.riemannian}
	For a  Riemannian $3$-dimensional Levi-Civita [1-1-1] metric $g$ we find
	local coordinates $(x_1^1,x_2^1,x_3^1)=(x^1,x^2,x^3)$ such that
	\[
		g = |X_1-X_2|\,|X_1-X_3| \,(dx^1)^2
			+|X_2-X_1|\,|X_2-X_3|\,(dx^2)^2
			+|X_3-X_1|\,|X_3-X_2|\,(dx^3)^2
	\]
	where $X_i=X_i(x^i)$ are univariate functions, which are also the
	eigenvalues of \eqref{eqn:benenti.tensor}.
	Reordering the coordinates, we may w.l.o.g.\ assume $X_1>X_2>X_3$.
	The above metric $g$ appears for the first time
	in~\cite{levi-civita_1896}.
	Some differential projective aspects of such metric is discussed in
	\cite{solodovnikov_1956,matveev_2012rel}.
\end{example}

For Riemannian Levi-Civita metrics in arbitrary dimension $n\geq3$ that admit
a projective vector field $v$, \cite{solodovnikov_1956} provides a collection
of defining ODEs under the requirement that the metric consists of solely
1-dimensional blocks, i.e.\ the metrics of
Example~\ref{ex:LC.metrics.riemannian} and their higher-dimensional
counterparts.

\subsection{Projective vector fields of Levi-Civita metrics}

The following lemma can be obtained from \cite{bolsinov_2011,bolsinov_2015cproj,solodovnikov_1956}.

\begin{lemma}
\label{lemma:eigenvalues.block.diagonals}
Let $g$ be an $n$-dimensional Levi-Civita metric of the form
\eqref{eqn:g1.LC.general} and not of constant curvature.
Let $v$ be a projective vector field for $g$.
Then the components of $v$, in the system of coordinates ~\eqref{eqn:local.coordinates.ordered}, are
\[
 \left(v_1(x_1^1),\dots,v_r(x_r^1),v_{r+1}(x_{r+1}),\dots,v_m(x_m)\right)
\]
where, for $1\leq i\leq r$, $v_i$ are univariate functions of variable $x_i^1$ and, for $i>r$,
$v_i=(v_i^1,\dots,v_i^{k_i})$
depend on the coordinates $x_i=(x_i^1,\dots,x_i^{k_i})$.
\end{lemma}

\begin{proof}
	Let $\bar g$ be the metric \eqref{eqn:g2.LC.general} and let $L=L(g,\bar
	g)$.
	Since $g$ is a Levi-Civita metric, $L$ is diagonal.
	Let~$i$ and~$j$ be indices such that $x^i$ and $x^j$ are coordinates
	on different eigendistributions, i.e.\ such that
	$\frac{\partial}{\partial x^i}$ and
	$\frac{\partial}{\partial x^j}$ belong to different eigenspaces of $L$.
	We denote the associated eigenvalues by $\lambda_k, \lambda_{\ell}$,
	respectively, with $\lambda_k\ne\lambda_{\ell}$.
	Because $L$ is diagonal, we infer $L\indices{^i_j}=0$.
	For the entry on the $i$-th row and $j$-th column
	of Equation~\eqref{eqn:LvL.solpol}, we therefore find:
	\[
		0 = \lie_vL\indices{^i_j}
		=
		v^cL\indices{^i_{j,c}}
		- v\indices{^i_{,c}}\,L\indices{^c_j}
		+ v\indices{^c_{,j}}\,L\indices{^i_c}
		=
		(\lambda_k-\lambda_{\ell})\,v\indices{^i_{,j}}\,,
	\]
	where we use that $L$ is diagonal.
	Since $\lambda_k-\lambda_{\ell}\ne0$, we conclude $v\indices{^i_{,j}}=0$.
	Repeating this for all pairs $(i,j)$ as above, we confirm the assertion.
\end{proof}

\begin{example}\label{ex:v.block}
	Let us consider the metric~\eqref{eqn:111.1} of non-constant curvature.
	According to Lemma~\ref{lemma:eigenvalues.block.diagonals}, a projective
	vector field $v$ for this metric has the following components:
	\[
		\big(v_1(x_1^1),v_2(x_2^1),v_3(x_3^1)\big)
		= \big(v^1(x),v^2(y),v^3(z)\big)\,.
	\]
	For the metric \eqref{eqn:21.before.1}, of non-constant curvature, a
	projective vector field $v$ has the following components:
	\[
		\big(v_1(x_1^1),v_2(x_2^1,x_2^2)\big)\,.
	\]
	In view of the consideration after equations \eqref{eqn:21.before.1}, for
	the metric~\eqref{eqn:21.1}, of non-constant curvature, a projective
	vector field $v$ has the following components:
	\[
		\big(v^1(x,y),v^2(x,y),v^3(z)\big)\,.	
	\]
\end{example}


\section{Levi-Civita metrics of type [1-1-1]}\label{sec:111}


We begin our discussion with the case of Levi-Civita metrics of type [1-1-1],
i.e.\ metrics that may be put into the form~\eqref{eqn:111.1} at least in a
small neighborhood around almost every point. This case includes Riemannian
metrics of the form~\eqref{eqn:111.1}. In the current section, we admit
Riemannian as well as Lorentzian signature.
The partner metric for \eqref{eqn:111.1} is \eqref{eqn:111.2} such that the
eigenvalues of $L=L(g,\bar g)$ are given by three univariate functions, more
precisely
\begin{equation}\label{eqn:L.diagonal.1.1.1}
	L = X_1(x^1)\,\partial_{x^1}\otimes dx^1
		+ X_2(x^2)\,\partial_{x^2}\otimes dx^2
		+ X_3(x^3)\,\partial_{x^3}\otimes dx^3
\end{equation}

\begin{remark}\label{rmk:transformations.111}
	The form of the metric~\eqref{eqn:111.1} is preserved under any
	simultaneous affine transformation of the coordinates,
	\[
		x^i_\text{new}=\mu x^i_\text{old}+\kappa_i\,,\quad
		\mu,\kappa_i\in\R\,,\ \mu\ne0\,,
	\]
	if we simultaneously redefine
	\[
		X_i^\text{new}(x^i_\text{new})
		= \frac{1}{|\mu|}\ X_i^\text{old}\left(
				\frac{x^i_\text{new}-\kappa_i}{\mu}
			\right)+t\,,
	\]
	with $t\in\R$.
	Note that, in particular, we may therefore transform
	$L_\text{new}=L_\text{old}-t\Id$, even without any modification of the
	coordinate system. Indeed, such a transformation corresponds with
	choosing a different metric $\bar g$ (respectively $\bar\sigma$
	via~\eqref{eqn:sigma.in.terms.of.g}) projectively equivalent, but
	non-proportional, to $g$. This even allows us to transform, without
	changing the metric $g$,
	\[
		L_\text{new}=\kappa L_\text{old}-t\Id
	\]
	for $\kappa,t\in\R$, $\kappa\ne0$.
\end{remark}


\noindent Equation~\eqref{eqn:Andreas.Solod.1}, in the [1-1-1] case, provide
a set of ODEs for the eigenvalues of \eqref{eqn:L.diagonal.1.1.1}, i.e.\ for
$X_1$, $X_2$ and $X_3$. More precisely, we have the following proposition.

\begin{lemma}\label{la:Solod.eq.1}
Let $g$ be a Levi-Civita metric of type [1-1-1] admitting a projective vector
field $v$. Then $v=v^i\partial_{x^i}$ where $v^i=v^i(x^i)$ are univariate
functions. Moreover the following system of ODEs
is satisfied:
\begin{subequations}\label{eqn:ODEs.111}
\begin{align}
	v^i \frac{dX_{i}}{dx^i}
	&= -b\,X^2_i
	+ (d-a)\,X_i
	+ c
	\label{eqn:ODEs.111.first}
	\\
	\frac{dv^{i}}{dx^i}
	&= -(a+d)
	\label{eqn:ODEs.111.second}
\end{align}
\end{subequations}
\end{lemma}
\begin{proof}
In view of the first part of Example \ref{ex:v.block}, the components $v^i$
of the projective vector field $v$ are such that $v^i=v^i(x^i)$. Then
\eqref{eqn:ODEs.111.first} is nothing but \eqref{eqn:Andreas.Solod.1} with
$\lambda=X_i(x^i)$.
Resubstituting \eqref{eqn:ODEs.111.first} into \eqref{eq:Lvgg}, we can solve
the obtained system w.r.t.\ the first derivatives of the component $v^i$ of
the projective vector field $v$, thus obtaining ~\eqref{eqn:ODEs.111.second}.
\end{proof}

\noindent The following proposition shows that all projective symmetries
Levi-Civita metrics of type [1-1-1] arise from homotheties of 1-dimensional
metrics.
\begin{proposition}\label{prop:splitting.gluing.111}

\

	\noindent(i) Let $g$ be a Levi-Civita metric of type [1-1-1] admitting a
	projective vector field $v= v^i\partial_{x^i}$, $v^i= v^i(x^i)$.
	Then $v$ is a homothetic vector field for the 1-dimensional metric
	$(dx^i)^2$.
	
	\noindent(ii) For $i\in\{1,2,3\}$, let $h_i$ be 1-dimensional metrics.
	Assume $u^i\partial_{x^i}\ne0$ (no summation), $u^i=u^i(x^i)$, is a homothetic vector field for $h_i$,
	$\lie_{u^i}h_i=-2Bh_i$ for a common constant $B\in\R$.
	Let $X_1,X_2,X_3$ be three functions satisfying
	\[
		(-2Bx^i+k_i) \frac{dX_{i}}{dx^i}
		= -b\,X^2_i
		+ \mu\,X_i
		+ c
	\]
	for $k_i\in\R$ and common constants $b,c,\mu\in\R$.
	Then the $X_i$ define a [1-1-1]-type Levi-Civita metric~\eqref{eqn:111}
	with projective vector field $v=u^i\partial_{x^i}$.
\end{proposition}
\begin{proof}
	For part (i), consider $h=(dx^i)^2$ and $u=f(x^i)\partial_{x^i}$. Then $u$ is
	homothetic for $h$ if
	\[
		\lie_uh = -2f'(x^i)h = -2Bh
	\]
	for constant $B\in\R$.
	The claim then follows from Equation~\eqref{eqn:ODEs.111.second}.
	
	Part (ii) follows immediately From Lemma \ref{la:Solod.eq.1}.
\end{proof}

	In \cite{solodovnikov_1956} formulas \eqref{eqn:ODEs.111.first} and
	\eqref{eqn:ODEs.111.second} are obtained in the Riemannian case.
	We compare the results. In the reference, two branches are distinguished:
	\begin{enumerate}
		\item
		If $v$ is an \emph{essential projective vector field}, then $b\ne0$.
		We then rescale $v$ by $-b$ and find the projective vector field
		$w=-\frac{v}{b}$.
		Now, define $f_i=X_i+\frac{a}{b}$.
		We then obtain
		from~\eqref{eqn:ODEs.111} (no summation convention applies)
		\begin{subequations}\label{eqn:solodovnikov.essential}
			\begin{align}
				\label{eqn:solodovnikov.essential.1}
				w^i\,\frac{df_i}{d x^i}
				&= f_i^2 + \alpha_1f_i + \alpha_0 \\
				\label{eqn:solodovnikov.essential.2}
				\frac{dw^i}{d x^i}
				&= -\alpha_1\,,
			\end{align}
		\end{subequations}
		where we define
		$$
		\alpha_1=-\frac{a+d}{b}\,,\quad
		\alpha_0=\frac{ad-cb}{b^2}\,.
		$$
		This system corresponds with Equations~(5.5) of
		\cite{solodovnikov_1956}.
		\item
		If $v$ is a \emph{homothetic projective vector field}, then $b=0$.
		In order to compare \eqref{eqn:ODEs.111} with the reference, we
		introduce $\eta\in\mathbb{R}$ such that $\lie_vg=2\eta\,g$.
		Using Equation~\eqref{eqn:Lv.sigma}, we then find
		$\lie_v\sigma=-\frac{\eta}{2}\sigma$ and conclude $\eta=-2a$.
		Equations \eqref{eqn:ODEs.111} thus take the form (no summation
		convention)
		\begin{subequations}
		\label{eqn:our.solodovnikov.homothetic}
		\begin{align}
			\label{eqn:our.solodovnikov.homothetic.1}
			v^i\,\frac{d X_i}{d x^i}
			&= (d-a)\,X_i + c \\
			\label{eqn:our.solodovnikov.homothetic.2}
			\frac{d v^i}{d x^i}
			&= -(a+d)\,.
		\end{align}
		\end{subequations}
		If $v$ is Killing, we have $a=0$. If $v$ is properly homothetic, the
		have $a\ne0$ and by rescaling $v$ we can fix $a$ to be any non-zero
		constant. Therefore, the number of free parameters is two.
		Note that \eqref{eqn:our.solodovnikov.homothetic} correspond with
		Equations~(5.8) of~\cite{solodovnikov_1956}.
	\end{enumerate}

\begin{remark}\label{rmk:compare.solodovnikov}
We emphasise the following difference between the two cases: In the
homothetic case, we may immediately take a solution
to~\eqref{eqn:our.solodovnikov.homothetic} and use it to write down the
metrics~\eqref{eqn:111.1} and~\eqref{eqn:111.2}.
In the essential case, for solutions to~\eqref{eqn:solodovnikov.essential},
this is only possible if $\lie_v\sigma$ is invertible. Otherwise, one of the
$f_i$ is going to be zero.
Indeed, if $\det(a\sigma+b\bar\sigma)=0$, then $-\frac{a}{b}$ is an
eigenvalue of $L$, and so $f_i=X_i+\frac{a}{b}=-\frac{a}{b}+\frac{a}{b}=0$
for some $i$ (see Remark \ref{rmk:transformations.111}).
Note that $g=\sum_idx_i^2\prod_{j\ne i}(X_i-X_j)=\sum_idx_i^2\prod_{j\ne
	i}(f_i-f_j)$. However, \eqref{eqn:111.2} cannot be computed for $f_i$.
Instead, we obtain the metrics projectively equivalent to $g$ by computing
\[
  \bar g = \mu\,\left(
	\pm\tfrac{(f_1-f_2)(f_1-f_3)}{(f_1+\ell)^2(f_2+\ell)(f_3+\ell)}\,(dx^1)^2
	\pm\tfrac{(f_2-f_1)(f_2-f_3)}{(f_1+\ell)(f_2+\ell)^2(f_3+\ell)}\,(dx^2)^2
	\pm\tfrac{(f_3-f_1)(f_3-f_2)}{(f_1+\ell)(f_2+\ell)(f_3+\ell)^2}\,(dx^3)^2
	\right)
\]
wherever defined (here $\ell$ and $\mu\ne0$ are real constants).
This formula can also be found in \cite[Eq. (26$'$)]{levi-civita_1896}
\end{remark}

The following two lemmas are the foundation for a full description of
[1-1-1]-type Levi-Civita metrics with projective vector fields.
We begin with the solution of the system~\eqref{eqn:solodovnikov.essential},
which yields essential projective vector fields\footnote{%
	We remark that, in the statement of the lemma, the condition
	$\alpha_1^2-4\alpha_0= 0$ is equivalent to requiring that the eigenvalues
	of	\eqref{eq:A} be coincident. Indeed $b^2(\alpha_1^2-4\alpha_0) =	
	(\kappa_1-\kappa_2)^2$, where $\kappa_i$ are the eigenvalues of
	\eqref{eq:A}.}

\begin{lemma}\label{la:explicit.solutions.111.essential}
The general solution to system \eqref{eqn:solodovnikov.essential}, up to a translation of $x^i$, is described below.
\begin{enumerate}
\item $\alpha_1\neq 0$
	\begin{enumerate}
		\item $\alpha_1^2-4\alpha_0>0$
		\begin{equation}\label{eqn:fi:1}
		w^i=-\alpha_1x^i\,,\quad f_i=\frac12\tanh\left(\frac{\ln(c_i|x^i|)}	 {2\alpha_1}\sqrt{\alpha_1^2-4\alpha_0}\right)\sqrt{\alpha_1^2-4\alpha_0}-\frac12
	 \alpha_1\,, \quad c_i\in\mathbb{R}, c_i>0\,.
		\end{equation}
		\item $\alpha_1^2-4\alpha_0= 0$
		\begin{equation}\label{eqn:fi:2}
		w^i=-\alpha_1x^i\,,\quad f_i=\alpha_1\left(
		\frac{1}{\ln(c_i|x^i|)}-\frac{1}{2}\right)\,,\quad c_i\in\mathbb{R},
		c_i>0\,.
		\end{equation}
		\item $\alpha_1^2-4\alpha_0<0$
		\begin{equation}\label{eqn:fi:3}
		w^i=-\alpha_1x^i\,,\quad f_i=-\frac12\tan\left(\frac{\ln(c_i|x^i|)}	 {2\alpha_1}\sqrt{-\alpha_1^2+4\alpha_0}\right)\sqrt{-\alpha_1^2+4\alpha_0}-\frac12
	 \alpha_1\,, \quad c_i\in\mathbb{R}, c_i>0\,.
		\end{equation}
		\end{enumerate}
	\item $\alpha_1=0$
		\begin{enumerate}
		\item
		$\alpha_0>0$. In this case
		$$
		w^i=c_i\,,\quad f_i=\sqrt{\alpha_0}\tan\left( \frac{\sqrt{\alpha_0}\,x^i}{c_i} \right)\,,\quad  c_i\in\mathbb{R}\,,\,\,c_i\neq 0\,.
		$$
		\item\label{item:La10.2b} $\alpha_0=0$. In this case
		$$
		w^i=c_i\,,\quad f_i=-\frac{c_i}{x^i}\,,\quad c_i\in\mathbb{R}\,.
		$$
		\item\label{item:La10.2c}
		$\alpha_0<0$. In this case, either
		$$
		\begin{array}{ll}
		w^i=c_i\,,\quad f_i=-\sqrt{-\alpha_0}\tanh\left( \frac{\sqrt{-\alpha_0}\,x^i}{c_i}\right), & c_i\in\mathbb{R}\,,\,\,c_i\neq 0
		\\
\text{or}
\\
		w^i=0\,,\quad f_i=\mp\sqrt{-\alpha_0} &
		\end{array}
		$$
\end{enumerate}
\end{enumerate}
\end{lemma}
\begin{proof}
All the involved ODEs are of Riccati type, which can be straightforwardly
solved.
%
\end{proof}

The following lemma covers homothetic vector fields of [1-1-1]-type
Levi-Civita metrics.

\begin{lemma}\label{la:explicit.solutions.111.homothetic}
The general solution to system \eqref{eqn:our.solodovnikov.homothetic}, up to a translation of $x^i$, is described below.
\begin{enumerate}
\item $a+d\neq 0$.
	\begin{enumerate}
	\item\label{item:La11.1a} $a-d\neq 0$.

	In this case \eqref{eqn:our.solodovnikov.homothetic.2} gives
	$$
	v^i=-(a+d)x^i\,.
	$$
	By substituting it into \eqref{eqn:our.solodovnikov.homothetic.1} we obtain an ODE whose general solution is
	$$
	 X_i=\frac{c}{a-d} + k_i |x^i|^{\frac{a-d}{a+d}}\,,\,\,k_i\in\R\,.
	$$
	\item\label{item:La11.1b} $a-d=0$, $a\neq0$.

	In this case \eqref{eqn:our.solodovnikov.homothetic.2} gives
	$$
		v^i=-2ax^i\,.
	$$
	By substituting it into \eqref{eqn:our.solodovnikov.homothetic.1} we
	obtain an ODE whose general solution is
	$$
		X_i=-\frac{c\ln(|x^i|)}{2a}+k_i\,,\quad k_i\in\R
	$$
	\end{enumerate}
\item $a+d=0$.
	\begin{enumerate}
	\item $a\neq0$.

In this case, either
$$
v^i=0\,,\quad X_i=\frac{c}{2a}
$$
or
$$
v^i=k_i\,,\quad X_i=\frac{c}{2a} + h_ie^{-\frac{2a}{k_i}x^i}\,,\quad k_i\in\R\,,\,\,k_i\neq0\,,\quad h_i\in\{-1,0,1\}\,.
$$
\item\label{item:La11.2b} $a=0=d$, $c=0$.
Then either
$$
v^i=0\,,\,\,\,X_i=X_i(x^i) \text{  is an arbitrary non-constant function of
$x^i$}
$$
or
$$
v^i=k_i\,,\,\,X_i=h_i\,,\,\,k_i\in\R\,,\,k_i\neq0\,,\,\,h_i\in\R
$$

\item\label{item:La11.2c} $a=0=d$, $c\neq 0$

Then
$$
v^i=k_i\,,\quad X_i=\frac{c}{k_i}x^i\,,\quad k_i\in\R\,,\,\,k_i\neq0
$$
\end{enumerate}
\end{enumerate}
\end{lemma}
\begin{proof}
The claim is straightforwardly obtained in analogy to
Lemma~\ref{la:explicit.solutions.111.essential}.
\end{proof}

Recall that if an eigenvalue of $L=L(g,\bar g)$ of the
form~\eqref{eqn:L.diagonal.1.1.1} is constant,
i.e.~$X_i(x^i)=\text{constant}$ for some $i$, then $\partial_{x^i}$ is a
Killing vector field as the metric~\eqref{eqn:111} does not depend on $x^i$
in that case.
Moreover, note that by definition, for a metric of [1-1-1] type, no
eigenvalue of $L$ can have multiplicity greater than one.
Finally, if all three eigenvalues of $L$ are constant, the
metric~\eqref{eqn:111} has constant curvature, and we are therefore going to
assume in the following that at least one eigenvalue of $L$ is non-constant.
\medskip

\noindent Note that if $L=L(\sigma,\bar\sigma)$ has a constant eigenvalue, we 
may w.l.o.g.\ change $\bar\sigma$ such that~$X_1=0$. Also, for a
metric~\eqref{eqn:111}, $L$ cannot have repeated eigenvalues. If $L$ has two
constant eigenvalues, we may in this case w.l.o.g.\ assume
$X_1=0,X_2=\rho\ne0$ by reordering the coordinates.
\medskip

For the following proposition, also recall the coordinate transformations
outlined in Remark~\ref{rmk:transformations.111}.

\begin{proposition}\label{prop:111.2constantEVs}
	Consider a metric~\eqref{eqn:111} of non-constant curvature such that
	$L(g,\bar g)$ has two constant eigenvalues.
	Then the projective algebra is
	$\projalg(g)=\la\partial_{x^1},\partial_{x^2}\ra$ and coincides with the
	Killing algebra.
\end{proposition}
\begin{proof}
	In Lemma~\ref{la:explicit.solutions.111.homothetic} only the
	case~\ref{item:La11.2b} contains solutions compatible with $L$ having two
	constant eigenvalues. We observe that this case we have the two
	trivial Killing vector fields $\partial_{x^1}$ and $\partial_{x^2}$.
	The metric in case~\ref{item:La10.2c} of
	Lemma~\ref{la:explicit.solutions.111.essential} is a special case of the
	aforementioned, and by virtue of the transformation in
	Remark~\ref{rmk:transformations.111} we obtain, in new coordinates $(x,y,z)$, the metric ($k\ne0$)
	\begin{equation*}
		g = \pm 2\left(1-\tanh\left(\frac{z}{k}\right)\right)\,dx^2
		\pm 2\left(1+\tanh\left(\frac{z}{k}\right)\right)\,dy^2
		\pm\left(1-\tanh\left(\frac{z}{k}\right)\right)\,
		\left(1+\tanh\left(\frac{z}{k}\right)\right)\,dz^2\,.
	\end{equation*}
	It is easily checked that this metric has constant curvature.
	The claim then follows straightforwardly.
\end{proof}
The proof of the following proposition is analogous to the previous one.
\begin{proposition}\label{prop:111.1constantEVs}
	Consider a metric~\eqref{eqn:111} of non-constant curvature such that
	$L(g,\bar g)$ has exactly one constant eigenvalue.
	Then $\projalg(g)=\la\partial_{x^1}\ra$, or there is a coordinate
	transformation to new coordinates $(x,y,z)$ that identifies $g$ as one of
	the following:
	\begin{enumerate}
		\item\label{item:1constEV.2Dalgebra.h}
			The metric
			\[
				g=\pm k_2k_3\,y^hz^h\,dx^2
					\pm k_2y^h(k_2y^h-k_3z^h)\,dy^2
					\pm k_3z^h(k_3z^h-k_2y^h)\,dz^2
			\]
			where $h\not\in\{-1,0,1\}, k_i\in\R, k_i\ne0$. Its projective
			algebra is
			generated by
			\[
				\partial_x\quad\text{and}\quad
				x\partial_x+y\partial_y+z\partial_z\,.
			\]
		\item\label{item:1constEV.3Dalgebra.-1}
			The metric
			\[
				g=\pm k_2k_3\,y^{-1}z^{-1}\,dx^2
					\pm k_2y^{-1}(k_2y^{-1}-k_3z^{-1})\,dy^2
					\pm k_3z^{-1}(k_3z^{-1}-k_2y^{-1})\,dz^2
			\]
			($k_i\in\R$, $k_i\ne0$) whose projective algebra is generated by
			\[
				\partial_x\,,\quad
				k_2\partial_y+k_3\partial_z
				\quad\text{and}\quad
				x\partial_x+y\partial_y+z\partial_z\,.
			\]
		\item\label{item:1constEV.2Dalgebra.exp}
			The metric
			\begin{equation*}
				g =\pm h_2h_3 e^{\frac{y}{k_2}}
							e^{\frac{z}{k_3}}\,dx^2
					\pm h_2e^{\frac{y}{k_2}}\left(
							h_2e^{\frac{y}{k_2}}
							-h_3e^{\frac{z}{k_3}}
						\right)\,dy^2
					\pm h_3e^{\frac{z}{k_3}}\left(
							h_3e^{\frac{z}{k_3}}
							-h_2e^{\frac{y}{k_2}}
						\right)\,dz^2
			\end{equation*}
			where $h_i\in\{\pm1\}$, $k_i\in\R$, $k_i\ne0$ and $k_2\ne\pm k_3$.
			Its projective algebra is generated by
			\[
				\partial_x\quad\text{and}\quad
				k_2\partial_y+k_3\partial_z\,.
			\]
		\item\label{item:1constEV.2Dalgebra.tanh}
			The metric
			\begin{align*}
				g &= \pm \left(\varepsilon -\tanh\left(\tfrac{y}{k_2}\right) \right)
						\left(
							\varepsilon-\tanh\left(\tfrac{z}{k_3}\right)
						\right)\,dx^2
					\\ &\quad
					\pm \left( \tanh\left(\tfrac{y}{k_2}\right)-\varepsilon\right)
					\left(
						\tanh\left(\tfrac{y}{k_2}\right)
						-\tanh\left(\tfrac{z}{k_3}\right)
					\right)\,dy^2
					\\ &\quad
					\pm\left( \tanh\left(\tfrac{z}{k_3}\right)-\varepsilon \right)
					\left(
						\tanh\left(\tfrac{z}{k_3}\right)
						-\tanh\left(\tfrac{y}{k_2}\right)
					\right)\,dz^2
			\end{align*}
			with $k_i\in\R$, $k_i\ne0$, $k_3\ne\pm k_2$, $\varepsilon\in\{\pm 
			1\}$.
			Its projective algebra is generated by
			\[
				\partial_x\quad\text{and}\quad
				k_2\partial_y+k_3\partial_z\,.
			\]
	\end{enumerate}
\end{proposition}
\begin{proof}
	In Lemma~\ref{la:explicit.solutions.111.essential} only the
	cases~\ref{item:La10.2b} and~\ref{item:La10.2c} can lead to metrics $g$
	with $L(g,\bar g)$ having exactly one constant eigenvalue.
	In Lemma~\ref{la:explicit.solutions.111.homothetic} only the
	cases~\ref{item:La11.1a}, \ref{item:La11.2b} and~\ref{item:La11.2c} can
	be of this kind.

	All these cases are special cases of~\ref{item:La11.2b} of
	Lemma~\ref{la:explicit.solutions.111.homothetic}, which is the generic
	case. Generically, according to
	Lemma~\ref{la:explicit.solutions.111.homothetic}, we have the Killing
	vector field $\partial_{x^i}$.
	Let us now consider the remaining, non-generic cases.
	The first case is case~\ref{item:La11.1a} from
	Lemma~\ref{la:explicit.solutions.111.homothetic}, assuming $h\ne1$.
	Note that for $h=1$, this metric coincides with case~\ref{item:La11.2c}
	of the same Lemma. In this case we therefore have a larger projective
	algebra.
	The remaining cases follow in an analogous manner.	
\end{proof}

\begin{proposition}\label{prop:111.0constantEVs}
	Consider a metric~\eqref{eqn:111} of non-constant curvature such that
	$L(g,\bar g)$ has no constant eigenvalue.
	Assume that $g$ admits a non-vanishing projective vector field.
	Then, after a change to new coordinates $(x,y,z)$, the metric is
	one of the following.
	\begin{enumerate}
		\item
		The metric
		\begin{align*}
			g &= \pm(k_1|x|^h-k_2|y|^h)(k_1|x|^h-k_3|z|^h)\,dx^2
				\pm(k_2|y|^h-k_1|x|^h)(k_2|y|^h-k_3|z|^h)\,dy^2
			\\ &\quad
				\pm(k_3|z|^h-k_1|x|^h)(k_3|z|^h-k_2|y|^h)\,dz^2
		\end{align*}
		where $h\not\in\{-1,0,1\}$, $k_i\in\R$, $k_i\ne0$.
		Its projective algebra is generated by
		\[
			x\partial_x+y\partial_y+z\partial_z\,.
		\]
		\item\label{item:1constEV.2Dalgebra.+1}
		The metric
		\begin{align*}
			g &= \pm(k_1x-k_2y)(k_1x-k_3z)\,dx^2
				\pm(k_2y-k_1x)(k_2y-k_3z)\,dy^2
				\\ &\quad
				\pm(k_3z-k_1x)(k_3z-k_2y)\,dz^2
		\end{align*}
		($k_i\in\R$, $k_i\ne0$) whose projective algebra is generated by
		\[
			\frac1{k_1}\,\partial_x
				+\frac1{k_2}\,\partial_y
				+\frac1{k_3}\,\partial_z
			\quad\text{and}\quad
			x\partial_x+y\partial_y+z\partial_z\,.
		\]
		\item\label{item:1constEV.2Dalgebra.-1}
		The metric
		\begin{align*}
			g &= \pm\left(\frac{k_1}{x}-\frac{k_2}{y}\right)
					\left(\frac{k_1}{x}-\frac{k_3}{z}\right)\,dx^2
				\pm\left(\frac{k_2}{y}-\frac{k_1}{x}\right)
					\left(\frac{k_2}{y}-\frac{k_3}{z}\right)\,dy^2
			\\ &\quad
				\pm\left(\frac{k_3}{z}-\frac{k_1}{x}\right)
					\left(\frac{k_3}{z}-\frac{k_2}{y}\right)\,dz^2
		\end{align*}
		($k_i\in\R$, $k_i\ne0$) whose projective algebra is generated by
		\[
			k_1\,\partial_x+k_2\,\partial_y+k_3\,\partial_z
			\quad\text{and}\quad
			x\partial_x+y\partial_y+z\partial_z\,.
		\]
		\item
		The metric
		\begin{equation*}
			g = \pm\ln\!\left(\tfrac{k_1|x|}{k_2|y|}\right)
					\ln\!\left(\tfrac{k_1|x|}{k_3|z|}\right)dx^2
					\pm \ln\!\left(\tfrac{k_2|y|}{k_1|x|}\right)
					\ln\!\left(\tfrac{k_2|y|}{k_3|z|}\right)dy^2
					\pm \ln\!\left(\tfrac{k_3|z|}{k_1|x|}\right)
					\ln\!\left(\tfrac{k_3|z|}{k_2|y|}\right)dz^2
		\end{equation*}
		($k_i\in\R$, $k_i>0$) whose projective algebra is generated by
		\[
			x\partial_x+y\partial_y+z\partial_z\,.
		\]
		\item
		The metric
		\begin{align*}
			g &= \pm\left(h_1e^{\tfrac{x}{k_1}}-h_2e^{\frac{y}{k_2}}\right)
					\left(h_1e^{\tfrac{x}{k_1}}-h_3e^{\frac{z}{k_3}}\right)\,dx^2
				\pm\left(h_2e^{\tfrac{y}{k_2}}-h_1e^{\frac{x}{k_1}}\right)
					\left(h_2e^{\tfrac{y}{k_2}}-h_3e^{\frac{z}{k_3}}\right)\,dy^2
				\\ &\quad
				\pm\left(h_3e^{\tfrac{z}{k_3}}-h_1e^{\frac{x}{k_1}}\right)
					\left(h_3e^{\tfrac{z}{k_3}}-h_2e^{\frac{y}{k_2}}\right)\,dz^2
		\end{align*}
		($h_i\in\{-1,1\}$, $k_i\in\R$, $k_i\ne0$) whose projective algebra is
		generated by
		\[
			k_1\partial_x+k_2\partial_y+k_3\partial_z\,.
		\]
		\item
		The metric
		\begin{align*}
			g &=	\pm\left(
						\tanh\left(\ln(k_1|x|^\beta)\right)
						-\tanh\left(\ln(k_2|y|^\beta)\right)
					\right)
					\left(
						\tanh\left(\ln(k_1|x|^\beta)\right)
						-\tanh\left(\ln(k_3|z|^\beta)\right)
					\right)\,dx^2
			\\ &\quad
				\pm\left(
					\tanh\left(\ln(k_2|y|^\beta)\right)
					-\tanh\left(\ln(k_1|x|^\beta)\right)
				\right)
				\left(
					\tanh\left(\ln(k_2|y|^\beta)\right)
					-\tanh\left(\ln(k_3|z|^\beta)\right)
				\right)\,dy^2
			\\ &\quad
				\pm\left(
					\tanh\left(\ln(k_3|z|^\beta)\right)
					-\tanh\left(\ln(k_1|x|^\beta)\right)
				\right)
				\left(
					\tanh\left(\ln(k_3|z|^\beta)\right)
					-\tanh\left(\ln(k_2|y|^\beta)\right)
				\right)\,dz^2
		\end{align*}
		($k_i\in\R$, $k_i>0$, $\beta\neq0$) whose projective algebra is
		generated by
		\[
			x\partial_x+y\partial_y+z\partial_z\,.
		\]
		\item
		The metric
		\begin{align*}
			g &= \pm\bigg(\frac{1}{\ln(k_1|x|)}-\frac{1}{\ln(k_2|y|)}\bigg)
				 \bigg(\frac{1}{\ln(k_1|x|)}-\frac{1}{\ln(k_3|z|)}\bigg)\,dx^2
				\\ &\quad
				\pm\bigg(\frac{1}{\ln(k_2|y|)}-\frac{1}{\ln(k_1|x|)}\bigg)
				 \bigg(\frac{1}{\ln(k_2|y|)}-\frac{1}{\ln(k_3|z|)}\bigg)\,dy^2
				\\ &\quad
				\pm\bigg(\frac{1}{\ln(k_3|z|)}-\frac{1}{\ln(k_1|x|)}\bigg)
				 \bigg(\frac{1}{\ln(k_3|z|)}-\frac{1}{\ln(k_2|y|)}\bigg)\,dz^2
		\end{align*}
		($k_i\in\R$, $k_i>0$) whose projective algebra is generated by
		\[
			x\partial_x+y\partial_y+z\partial_z\,.
		\]
		\item
		The metric
		\begin{align*}
			g &= \pm\left(
					\tan\left(\ln(k_1|x|^\beta)\right)
					-\tan\left(\ln(k_2|y|^\beta)\right)
				\right)
				\left(
					\tan\left(\ln(k_1|x|^\beta)\right)
					-\tan\left(\ln(k_3|z|^\beta)\right)
				\right)\,dx^2
				\\ &\quad
				\pm\left(
					\tan\left(\ln(k_2|y|^\beta)\right)
					-\tan\left(\ln(k_1|x|^\beta)\right)
				\right)
				\left(
					\tan\left(\ln(k_2|y|^\beta)\right)
					-\tan\left(\ln(k_3|z|^\beta)\right)
				\right)\,dy^2
				\\ &\quad
				\pm\left(
					\tan\left(\ln(k_3|z|^\beta)\right)
					-\tan\left(\ln(k_1|x|^\beta)\right)
				\right)
				\left(
					\tan\left(\ln(k_3|z|^\beta)\right)
					-\tan\left(\ln(k_2|y|^\beta)\right)
				\right)\,dz^2
		\end{align*}
		($k_i\in\R$, $k_i>0$, $\beta\neq0$) whose projective algebra is
		generated by
		\[
			x\partial_x+y\partial_y+z\partial_z\,.
		\]
		\item
		The metric
		\begin{align*}
			g &= \pm\left(
					\tan\left(\tfrac{x}{k_1}\right)
					-\tan\left(\tfrac{y}{k_2}\right)
				\right)
				\left(
					\tan\left(\tfrac{x}{k_1}\right)
					-\tan\left(\tfrac{z}{k_3}\right)
				\right)\,dx^2
				\\ &\quad
				\pm\left(
					\tan\left(\tfrac{y}{k_2}\right)
					-\tan\left(\tfrac{x}{k_1}\right)
				\right)
				\left(
					\tan\left(\tfrac{y}{k_2}\right)
					-\tan\left(\tfrac{z}{k_3}\right)
				\right)\,dy^2
				\\ &\quad
				\pm\left(
					\tan\left(\tfrac{z}{k_3}\right)
					-\tan\left(\tfrac{x}{k_1}\right)
				\right)
				\left(
					\tan\left(\tfrac{z}{k_3}\right)
					-\tan\left(\tfrac{y}{k_2}\right)
				\right)\,dz^2
		\end{align*}
		($k_i\in\R$, $k_i\ne0$) whose projective algebra is generated by
		\[
			k_1\partial_x+k_2\partial_y+k_3\partial_z\,.
		\]
		\item
		The metric
		\begin{align*}
			g &= \pm\left(
					\tanh\left(\tfrac{x}{k_1}\right)
					-\tanh\left(\tfrac{y}{k_2}\right)
				\right)
				\left(
					\tanh\left(\tfrac{x}{k_1}\right)
					-\tanh\left(\tfrac{z}{k_3}\right)
				\right)\,dx^2
				\\ &\quad
				\pm\left(
					\tanh\left(\tfrac{y}{k_2}\right)
					-\tanh\left(\tfrac{x}{k_1}\right)
				\right)
				\left(
					\tanh\left(\tfrac{y}{k_2}\right)
					-\tanh\left(\tfrac{z}{k_3}\right)
				\right)\,dy^2
				\\ &\quad
				\pm\left(
					\tanh\left(\tfrac{z}{k_3}\right)
					-\tanh\left(\tfrac{x}{k_1}\right)
				\right)
				\left(
					\tanh\left(\tfrac{z}{k_3}\right)
					-\tanh\left(\tfrac{y}{k_2}\right)
				\right)\,dz^2
		\end{align*}
		($k_i\in\R$, $k_i\ne0$) whose projective algebra is generated by
		\[
			k_1\partial_x+k_2\partial_y+k_3\partial_z\,.
		\]
	\end{enumerate}
\end{proposition}
\begin{proof}
	In view of Remark~\ref{rmk:transformations.111}, consider the solutions
	in Lemmas~\ref{la:explicit.solutions.111.essential}
	and~\ref{la:explicit.solutions.111.homothetic}.
	Consider the first case of
	Lemma~\ref{la:explicit.solutions.111.homothetic}.
	Without changing the metric, we translate the solutions
	$X_i=\frac{c}{a-d} + k_i |x^i|^{\frac{a-d}{a+d}}$ by the common constant
	$\frac{c}{a-d}$ and rename $h=\frac{a-d}{a+d}$.
	We thus arrive at the claim for $h\ne0$ (for $h=0$ the metric would have
	constant curvature). Note that the projective vector field, up to
	multiplication by a constant, is given by the components $v^i=x^i$.
	Consider the second case in
	Lemma~\ref{la:explicit.solutions.111.homothetic},
	$X_i=-\frac{c}{2a}\,\ln(|x^i|)+k_i$. By an affine transformation of the
	coordinates of the form $x^i\to \frac{2a}{c}x^i$, and using the standard
	identities for logarithms, we arrive at the claim.
	The remaining cases follow similarly, from the remaining cases of
	Lemma~\ref{la:explicit.solutions.111.homothetic} and the solutions from
	Lemma~\ref{la:explicit.solutions.111.essential}. Recall that no two
	eigenvalues of the Benenti tensor $L(g,\bar{g})$ are allowed to coincide.
\end{proof}

The previous three propositions answer the question which Levi-Civita metrics
of type [1-1-1] admit a projective symmetry algebra of at least dimension~1.
The following theorem summarises these results.

\begin{theorem}\label{thm:1.1.1}
	Consider a metric~\eqref{eqn:111} of non-constant curvature with a
	non-vanishing projective vector field.
	Then it is a metric in Proposition \ref{prop:111.2constantEVs}, 
	\ref{prop:111.1constantEVs} or \ref{prop:111.0constantEVs}. In particular:
	\begin{enumerate}[label=\alph*)]
		\item
		If the projective algebra is exactly 2-dimensional, it is as in
		Proposition~\ref{prop:111.2constantEVs}, or as in
		\ref{item:1constEV.2Dalgebra.h}, \ref{item:1constEV.2Dalgebra.exp} or
		\ref{item:1constEV.2Dalgebra.tanh}
		of Proposition~\ref{prop:111.1constantEVs}, or as in
		\ref{item:1constEV.2Dalgebra.+1} or \ref{item:1constEV.2Dalgebra.-1}
		of Proposition~\ref{prop:111.0constantEVs}.
		\item
		If the projective algebra is exactly 3-dimensional, it is as
		in~\ref{item:1constEV.3Dalgebra.-1} of
		Proposition~\ref{prop:111.1constantEVs}.
	\end{enumerate}
	If the projective algebra is at least $4$-dimensional, the metric is
	already of constant curvature.
\end{theorem}

\section{Levi-Civita metrics of type [2-1]}\label{sec:21}

In order to conclude our study of 3-dimensional Levi-Civita metrics, we now
focus on Levi-Civita metrics of type [2-1] of non-constant curvature that 
admit a projective vector field $v$. We consider solely the positive branch 
of \eqref{eqn:21.1},
\begin{equation}\label{eqn:g1.2-1.for.discussion}
	g = \zeta(z)(h+dz^2)\,,\quad \zeta(z)=Z(z)-\rho\,,
\end{equation}
because $\zeta(z)(h-dz^2)=-\zeta(z)(-h+dz^2)$.
For the projective vector field $v$, in view of Example \ref{ex:v.block}, we have the following expression:
\begin{equation}\label{eq:raw.degLC.v}
	v = u+ \alpha(z)\partial_z\,,
\end{equation}
with
\begin{equation}\label{eqn:proj.symmetry.u}
	u = v^1(x,y)\partial_x + v^2(x,y)\partial_y\,.
\end{equation}

Our aim is analogous to that in Section~\ref{sec:111}, where we have
considered Levi-Civita metrics of type~[1-1-1]: We shall determine the
functions $\zeta(z)$ such that~\eqref{eqn:g1.2-1.for.discussion} admits
non-vanishing projective symmetries.

\begin{remark}\label{rem:allowed.transf}
	The coordinate transformations that preserve the
	form of~\eqref{eqn:g1.2-1.for.discussion} are
	\begin{equation}\label{eqn:allowed.changes.coord}
		(x_\text{new},y_\text{new},z_\text{new})
		=\big(x_\text{new}(x_\text{old},y_\text{old}),
			y_\text{new}(x_\text{old},y_\text{old}),
			z_\text{new}(z_\text{old})=k_1z_\text{old}+k_2\big)\,,\,\,k_i\in\R\,.
	\end{equation}
In addition we can simultaneously transform
\[
	\rho_\text{new}=\rho_\text{old}+t\,,\quad
	Z_\text{new}=Z_\text{old}+t
\]
for constant $t\in\R$. This transformation does not
change~\eqref{eqn:g1.2-1.for.discussion}.
\end{remark}

We now describe first all those metrics~\eqref{eqn:g1.2-1.for.discussion}
that are of constant curvature and thus admit a 15-dimensional algebra of
projective symmetries.
For metrics~\eqref{eqn:g1.2-1.for.discussion} of non-constant curvature, we
are then going to show, in Section~\ref{sec:splitting.gluing}, that
projective vector fields arise from homothetic vector fields of $h$.
We begin by showing that if~\eqref{eqn:g1.2-1.for.discussion} is of constant
curvature, then also $h$ is.
\begin{lemma}\label{la:Rg.constant.Rh.constant}
	Let $g$ be a metric~\eqref{eqn:g1.2-1.for.discussion} of constant scalar
	curvature.
	Then $h$ has constant curvature.
\end{lemma}
\begin{proof}
	Let us denote the scalar curvature of $g$ and $h$ by, respectively, $R_g$
	and
	$R_h$.
	It is easily verified that
	$$
	R_g(x,y,z) = \frac{R_h(x,y)}{\zeta(z)} + \frac1{2\zeta(z)^3}\,\left(
	3\,\left(\frac{d\zeta(z)}{dz}\right)^2
	- 4\,\zeta(z)\frac{d^2\zeta(z)}{dz^2}
	\right)
	$$
	In view of the above formula, since $R_g$ is constant, we have that
	$$
	\frac{\partial R_g}{\partial x}=\frac{\partial R_g}{\partial
	y}=0=\frac{\partial R_h}{\partial x}=\frac{\partial R_h}{\partial y}\,.
	$$
	Thus, $R_h$ is a constant.
\end{proof}

The following lemma allows us simplify~\eqref{eqn:g1.2-1.for.discussion} if
$h$ has constant curvature.

\begin{lemma}\label{la:normalise.curvature}
Let $g$ be a metric \eqref{eqn:g1.2-1.for.discussion} such that $h$ has
constant curvature equal to $\kappa\neq 0$. Then such a metric is isometric to
\begin{equation}\label{eqn:metric.w.l.o.g}
	\tilde{\zeta}(\tilde{z})(\tilde{h}+d\tilde{z}^2)\,,\quad
	\tilde{\zeta}(\tilde{z}):=\frac{1}{|\kappa|}\,
					\zeta\left(\frac{\tilde{z}}{\sqrt{|\kappa|}}\right)
\end{equation}
with $\tilde{h}$ a 2-dimensional metric with constant curvature equal to $\frac{\kappa}{|\kappa|}$.
\end{lemma}
\begin{proof}
The metric $h$ is locally isometric to $\frac{\tilde{h}}{|\kappa|}$ with $\tilde{h}$ having curvature equal to $\frac{\kappa}{|\kappa|}$. Then, by considering the transformation $z=\frac{\tilde{z}}{\sqrt{|\kappa|}}$, metric \eqref{eqn:g1.2-1.for.discussion} assumes the form \eqref{eqn:metric.w.l.o.g}.
\end{proof}

We are now able to formulate precisely which functions $\zeta$ lead to
constant curvature metrics.

\begin{lemma}\label{la:cc.21}
	Consider the metric~\eqref{eqn:g1.2-1.for.discussion}.
	If it has constant curvature, then by a change of coordinate of
	type~\eqref{eqn:allowed.changes.coord}, locally we achieve one of the
	following:
	\begin{enumerate}
		\item $g$ has zero curvature.
		\begin{enumerate}
			\item\label{item:cc.1a}
			$g = \varepsilon_0\,(\varepsilon_1dx^2+\varepsilon_2dy^2+dz^2)$.
			\item\label{item:cc.1b}
			$g = \frac{\varepsilon_0k}{z^2}\,(
					\varepsilon_1\,dx^2+\varepsilon_2\,dy^2+dz^2)$.
		\end{enumerate}
		\item $g$ has positive constant curvature.
		\begin{enumerate}
			\item
			$g = \frac{k}{\cosh^2(z)}\,(
					dx^2+\varepsilon_2\sin^2(x)dy^2+dz^2)$.
			\item
			$g = -\frac{k}{\cos^2(z)}\,(
					-dx^2+\varepsilon_2\sin^2(x)dy^2+dz^2)$.
			\item
			$g = \frac{k}{\cosh^2(z)}\,(
					-dx^2+\varepsilon_2\sinh^2(x)dy^2+dz^2)$.
			\item
			$g = -\frac{k}{\cos^2(z)}\,(
					dx^2+\varepsilon_2\sinh^2(x)dy^2+dz^2)$.
		\end{enumerate}
		\item\label{item:cc.3} $g$ has negative constant curvature.
		\begin{enumerate}
			\item
			$g = -\frac{k}{\cosh^2(z)}\,(
					dx^2+\varepsilon_2\sin^2(x)dy^2+dz^2)$.
			\item
			$g = \frac{k}{\cos^2(z)}\,(
					-dx^2+\varepsilon_2\sin^2(x)dy^2+dz^2)$.
			\item
			$g = -\frac{k}{\cosh^2(z)}\,(
					-dx^2+\varepsilon_2\sinh^2(x)dy^2+dz^2)$.
			\item\label{item:cc.3d}
			$g = \frac{k}{\cos^2(z)}\,(
					dx^2+\varepsilon_2\sinh^2(x)dy^2+dz^2)$.
		\end{enumerate}
	\end{enumerate}
	where $\varepsilon_i\in\{\pm1\}$ and $k\in\R$, $k>0$.
\end{lemma}
\begin{proof}
	We give the explicit proof for the case when $h$ is Riemannian. For the
	other signatures the proof is analogous.
	So assume that $h$ has signature $(++)$. Due to
	Lemma~\ref{la:Rg.constant.Rh.constant}, we have that $h$ is of
	constant curvature. In view of Lemma~\ref{la:normalise.curvature},
	w.l.o.g., we can
	suppose such curvature equal to $0$, $1$ or $-1$.
	
	\noindent$\bullet$ If $h$ has curvature equal to $0$, by a change of coordinates $(x,y)$
	we have
	\begin{equation}\label{eqn:zeta.x.y.z}
		g = \zeta(z)(dx^2+dy^2+dz^2)\,.
	\end{equation}
Now, let us suppose that \eqref{eqn:zeta.x.y.z} has constant curvature equal to $\kappa$. Thus (by cutting with the planes $\langle \partial_x\,,\partial_y\rangle$ and $\langle \partial_x\,,\partial_z\rangle$) we have that
\[
-\frac1{4\zeta(z)^3}\,\left(\frac{d\zeta(z)}{dz}\right)^2
= \frac{\left(\frac{d\zeta(z)}{dz}\right)^2- \zeta(z)\,\frac{d^2\zeta(z)}{dz^2}}{2\zeta(z)^3}
= \kappa\,.
\]
By solving the first equation we obtain
	\begin{equation}\label{eqn:zeta.right}
		\zeta(z)=\frac{\varepsilon}{(c_1z+c_2)^2}\,,\quad
		c_i\in\mathbb{R}\,,\,\,\varepsilon\in\{\pm1\}\,.
	\end{equation}
On the other hand, we observe that a metric \eqref{eqn:zeta.x.y.z} with
$\zeta(z)$ given by \eqref{eqn:zeta.right} has constant curvature equal to
$\pm c_1^2$.  If $c_1=0$ then we have $\zeta\in\R$ and we easily get the form
\ref{item:cc.1a} of the claim by rescaling the coordinates. If however
$c_1\ne0$, then we achieve the form~\ref{item:cc.1b}.

\medskip\noindent $\bullet$ If $h$ has curvature equal to $1$, then by a
	change of coordinates $(x,y)$  we have
	\begin{equation}\label{eqn:zeta.x.sinx.z}
		g = \zeta(z)(dx^2+\sin(x)^2dy^2+dz^2)
	\end{equation}
	Similarly as in the above case, if the metric \eqref{eqn:zeta.x.sinx.z} has constant  curvature $\kappa$ then
$$
-\frac{\left(\frac{d\zeta(z)}{dz}\right)^2-4\zeta(z)^2}{4\zeta(z)^3} =  \frac{\left(\frac{d\zeta(z)}{dz}\right)^2- \zeta(z)\,\frac{d^2\zeta(z)}{dz^2}}{2\zeta(z)^3}=\kappa
$$
and by solving the first equation we obtain
	\begin{equation*}
		\zeta(z)=\frac{\varepsilon}{(
			c_1\cosh(z)+c_2\sinh(z))^2
		}\,,\quad c_i\in\mathbb{R}\,,\,\,\varepsilon\in\{\pm1\}\,.
	\end{equation*}
The resulting metric \eqref{eqn:zeta.x.sinx.z} has constant curvature equal to $c_1^2-c_2^2$.
	Defining $c=\sqrt{|c_1^2-c_2^2|}$ and
	$\tanh(\gamma)=\frac{c_2}{c_1}$, we arrive at
	\begin{equation*}
		\zeta(z)=\frac{\varepsilon}{c\cosh(z+\gamma)^2}\,,
	\end{equation*}
	and after a translation of $z$ by $\gamma$ we arrive at the claim.
	
	\medskip\noindent $\bullet$ The case when $h$ has curvature equal
	to $-1$ is very similar to the above case: we eventually arrive to the
	form~\ref{item:cc.3d} of the claim. We omit the details.
\end{proof}

\subsection{A splitting-gluing result for projective vector fields of [2-1]-type Levi-Civita metrics}\label{sec:splitting.gluing}

Let us now restrict to Levi-Civita metrics of [2-1] type of non-constant
curvature that admit projective symmetries.
%
Recall that a projective vector field $v$ of non-constant curvature metrics
~\eqref{eqn:g1.2-1.for.discussion} is necessarily of the
form~\eqref{eq:raw.degLC.v}.
It is easy to see that $u$ is a projective vector field for $h$.
Indeed, by computing the geodesic equations \eqref{eqn:3.dim.proj.conn}, where the $f^k_{ij}$'s are given by \eqref{eqn:non.lin.syst},
with $\Gamma^k_{ij}$ the Christoffel symbols of the Levi-Civita connection of $g$, we realise that the first equation of system \eqref{eqn:3.dim.proj.conn} coincides with the $2$-dimensional projective connection associated to the metric $h$.
The following proposition refines this statement further. In preparation of it, recall Lemma \ref{la:eigenvalues.L.roots.solodovnikov}, which provides the following equations:
\begin{subequations}\label{eqn:solodovnikov.21}
	\begin{align}
		\label{eqn:solodovnikov.21.rho}
		0 &= b\rho^2-(d-a)\rho-c
		\\
		\label{eqn:solodovnikov.21.Z}
		\alpha\frac{dZ}{dz} &= -bZ^2+(d-a)Z+c
	\end{align}
\end{subequations}

\begin{proposition}\label{prop:proj.symmetries.descend.21type}~
	
	\noindent(i)
	Let $g$ be a $3$-dimensional Levi-Civita
	metric~\eqref{eqn:g1.2-1.for.discussion} of non-constant curvature
	with projective vector field~\eqref{eq:raw.degLC.v}.
	Then \eqref{eqn:proj.symmetry.u} is
	homothetic for the metric $h$, with
	\begin{equation}\label{eqn:Luh.Ch}
		\lie_uh=-Ch
	\end{equation}
	where
	\begin{equation}\label{eqn:C}
		C = 2b\rho + 3a + d.
	\end{equation}
	
	\noindent(ii)
	Conversely, let $u$ be a homothetic vector field for a
	$2$-dimensional metric $h$, such that \eqref{eqn:Luh.Ch} holds.
	Let $g$ be the metric~\eqref{eqn:g1.2-1.for.discussion}.
	Then the vector field \eqref{eq:raw.degLC.v} is a projective vector
	field for $g$ if and only if
	\begin{subequations}
	\begin{align}
		\label{eqn:alpha.prime}
		-2\alpha'(z) &= C+b\zeta(z) \\
		\label{eqn:alpha.zeta.prime}
		\alpha(z)\zeta'(z) &= -b\zeta^2(z)+\left(2B-C\right)\zeta(z)
	\end{align}
	\end{subequations}
	where
	\begin{equation}\label{eqn:B}
		B = a+d
	\end{equation}
	and $a,b,d$ are as in~\eqref{eq:A}.
\end{proposition}
\begin{proof}
	We begin with part (i) by computing the projections of \eqref{eq:Lvg1}
	onto the 2-dimensional component $M_2$ and the 1-dimensional component
	$M_1$, respectively, see the discussion after~\eqref{eqn:21}.
	Since $v$ is of the form \eqref{eq:raw.degLC.v}, we infer,
	\begin{equation}\label{eqn:decompose.Lvg}
		\lie_vg
		= \lie_v(\zeta(h+dz^2))
		= \zeta'\alpha\,h+\zeta\lie_uh+\alpha\zeta'dz^2+2\zeta\alpha'dz^2\,.
	\end{equation}
	Substituting this in \eqref{eq:Lvg1}, and inserting $n=3$ as well as
	$L=\rho(\partial_x\otimes dx+\partial_y\otimes
	dy)+(\zeta+\rho)\partial_z\otimes dz$, we arrive at
	\begin{align}
		\label{eqn:Luh.pre}
		\lie_uh &= \left(
						-\frac{\zeta'}{\zeta}\,\alpha
						-4a -b(4\rho+\zeta)
					\right)\,h
		\\
		\label{eqn:alpha.prime.pre}
		-2\alpha' &= \frac{\zeta'}{\zeta}\,\alpha
						+4a+2b\rho\zeta +4b\rho\,.
	\end{align}
	Combining the equations~\eqref{eqn:solodovnikov.21}, we
	obtain~\eqref{eqn:alpha.zeta.prime}.
	Reinserting~\eqref{eqn:alpha.zeta.prime} into~\eqref{eqn:Luh.pre}, we
	find~\eqref{eqn:Luh.Ch}.

	For part (ii) consider~\eqref{eqn:decompose.Lvg} and
	insert~\eqref{eqn:Luh.Ch}.
	Then~\eqref{eqn:alpha.prime} holds due to~\eqref{eqn:solodovnikov.21.Z}.
	Reinserting into~\eqref{eqn:alpha.prime.pre}, we find
	\eqref{eqn:alpha.zeta.prime}.
We compute $\lie_vg=\lie_v(\zeta(h+dz^2))$, but now for a given $h$ with
	$\lie_uh=-Ch$. We substitute this into~\eqref{eq:Lvg1} and take the
	projection onto the 2-dimensional and the 1-dimensional component. This
	yields~\eqref{eqn:alpha.prime}. Equation~\eqref{eqn:alpha.zeta.prime} is
	\eqref{eqn:solodovnikov.21} rewritten in terms of $\zeta$.
\end{proof}

\begin{remark}
In the hypotheses of Proposition \ref{prop:proj.symmetries.descend.21type}, vector field $u=v^1(x,y)\partial_x+v^2(x,y)\partial_y$ is a homothety for the metric $h$. In the case when $h$ is a Riemannian metric this implies that $u$ is a holomorphic vector field in the following sense: by working in conformal coordinates $(x,y)$, function $v^1(x,y)+iv^2(x,y)$ turns out to be holomorphic. In the case when $h$ is of Lorentzian signature, $u$ turns out to be a para-holomorphic vector field, i.e., by working in null coordinates $(x,y)$ (i.e., $h=e^{f(x,y)}dxdy$ for some function $f$), $v^1=v^1(x)$ and $v^2=v^2(y)$, see \cite{Alekseevsky_2009} and \cite{MANNO2012365} for more details.
\end{remark}

\begin{proposition}\label{prop:gluing}
	Let $h=h_{11}dx^2+2h_{12}dxdy+h_{22}dy^2$, $h_{ij}=h_{ij}(x,y)$, be a
	$2$-dimensional metric with homothetic vector field~$u$ that is nowhere
	vanishing such that $\lie_uh=-Ch$, $C\in\R$. Then
	the vector field \eqref{eq:raw.degLC.v} is
	projective for the metric $g=\zeta(z)\,(h+dz^2)$ if and only if $\zeta$
	and $\alpha$ satisfy the system
	\begin{subequations}\label{eqn:alpha.zeta.ODEs}
		\begin{align}
			\zeta (\alpha \zeta'' -\alpha'\zeta' - C \zeta') - \alpha {\zeta'}^2 &=0
			\label{eqn:alpha.zeta.ODEs.1}
			\\
			\zeta(-2\alpha''\zeta +\alpha \zeta'' + \alpha'\zeta' )- \alpha
			{\zeta'}^2&=0
			\label{eqn:alpha.zeta.ODEs.2}
		\end{align}
	\end{subequations}
	with $\zeta$ being a nowhere vanishing solution.
\end{proposition}
\begin{proof}
	In an appropriate system of coordinates  $(x,y)$,
	\begin{equation}\label{eqn:lie.metric}
		u=\partial_x\,,\quad
		h=e^{-C x}
		\left(
		\begin{array}{cc}
			E(y) & F(y)
			\\
			F(y) & G(y)
		\end{array}
		\right)\,,\quad C\in\R\,.
	\end{equation}
	A direct computation shows that the 18 PDEs forming System \eqref{eqn:cond.proj} reduce to only 2 independent conditions, namely
 \eqref{eqn:alpha.zeta.ODEs}.
\end{proof}

\begin{remark}
	For any $h$ of the form \eqref{eqn:lie.metric}, there exist non-zero functions $\zeta(z)$ and $\alpha(z)$ such that the vector field \eqref{eq:raw.degLC.v}, with $\lie_uh=-Ch$, is a projective vector field for the metric \eqref{eqn:g1.2-1.for.discussion}.
	For instance, if $C\ne0$, the functions
	\[
		\zeta=C z\,,\quad \alpha=-\frac{1}{2}Cz
	\]
	solve system \eqref{eqn:alpha.zeta.ODEs}.
	For the case $C=0$, the functions $\zeta=\frac{1}{z^2}$, $\alpha=\frac{1}{z}$ solve system \eqref{eqn:alpha.zeta.ODEs}, for example.
\end{remark}

\begin{remark}\label{rmk:partial.integration.alpha.zeta}
	Note that, in Proposition~\ref{prop:gluing}, Equations
	\eqref{eqn:alpha.zeta.ODEs}
	are consequences of
	\eqref{eqn:alpha.prime} and~\eqref{eqn:alpha.zeta.prime}.
	Indeed  \eqref{eqn:alpha.prime} gives
	\begin{equation}\label{eqn:b.Gianni}
		b=-\frac{2\alpha'+C}{\zeta}\,.
	\end{equation}
	Then differentiate \eqref{eqn:b.Gianni} once w.r.t.~$z$ to obtain the
	equation
	\begin{equation}\label{eqn:b.Gianni.2}
	2\zeta'\alpha'-2\zeta\alpha'' +C\zeta'=0\,.
	\end{equation}
	Likewise,
	solve~\eqref{eqn:alpha.zeta.prime} for $2B-C$, then differentiate w.r.t.~$z$,      to obtain:
\begin{equation}\label{eqn:blablabla}
\zeta(\alpha'\zeta'+\alpha\zeta'')+b\zeta'\zeta^2=\alpha\zeta'^2\,.
\end{equation}
	Inserting \eqref{eqn:b.Gianni} into the Equation \eqref{eqn:blablabla}, 
	we obtain \eqref{eqn:alpha.zeta.ODEs.1}. Equation 
	\eqref{eqn:alpha.zeta.ODEs.2} then follows from 
	\eqref{eqn:b.Gianni},\eqref{eqn:b.Gianni.2} and~\eqref{eqn:blablabla}.
	Conversely, note that from \eqref{eqn:alpha.zeta.ODEs} we obtain
	\eqref{eqn:b.Gianni.2}, which is equivalent to \eqref{eqn:b.Gianni} since
	$b$ an arbitrary constant.
\end{remark}

The following lemma covers the most basic situation, namely constant 
$\zeta(z)$.

\begin{corollary}\label{cor:constant.Z}
	Let $h$ be a $2$-dimensional metric. Let $g=k\,(h+dz^2)$ for some
	non-zero constant $k$.
	If~$g$ is of non-constant curvature, the projective symmetries of $g$ are
	the vector fields
	\[
		v = u+(k_1z+k_0)\partial_z\,,
	\]
	where $u$ is a homothetic vector field of $h$, and where $k_0,k_1\in\R$.
\end{corollary}
\begin{proof}
	Substituting $\zeta=k$ into the system~\eqref{eqn:alpha.zeta.ODEs} of
	Proposition~\ref{prop:gluing}, we arrive at $-2k^2\alpha''(z)=0$,
	i.e.~$\alpha''(z)=0$. This proves the claim.
\end{proof}

The following example illustrates the situation.
\begin{example}\label{ex:submaximal}
	Take the metric $h=dx^2+\sin^2(x)dy^2$, which has the homothetic vector
	fields (in fact, these are all Killing vector fields)
	\[
		u = k_0\partial_y
			+k_1\,\left( -\cos(y)\partial_x+\cot(x)\sin(y)\partial_y \right)
			+k_2\,\left( \sin(y)\partial_x+\cot(x)\cos(y)\partial_y \right)\,.
	\]
	The metric $g=dx^2+\sin^2(x)dy^2+dz^2$, on the other hand, admits the
	projective symmetry algebra
	\[
		u = k_0\partial_y
			+k_1\,\left( -\cos(y)\partial_x+\cot(x)\sin(y)\partial_y \right)
			+k_2\,\left( \sin(y)\partial_x+\cot(x)\cos(y)\partial_y \right)
			+k_3z\partial_z
			+k_4\partial_z\,.
	\]
\end{example}

\begin{remark}
A 3-dimensional Riemannian metric, which is not of constant curvature, has a projective algebra of dimension $\leq5$ \cite{kruglikov_2014}. The sub-maximal dimension $5$ is realised in Example \ref{ex:submaximal}, and it is realised also for $g=dx^2+\sinh^2(x)dy^2+dz^2$.

Note that for a 3-dimensional metric of Lorentz signature the sub-maximal dimension of the projective algebra is 6 instead. According to
\cite{kruglikov_2014}, see also the references therein, this is realised for the metrics ($k\neq0$)
\begin{align*}
	g_1 &= kdx^2+2(2-c)e^{cx}dxdy+e^{2x}dz^2\,,\quad c\notin\{1, 2\} \\
	g_2 &= kdx^2+e^{2x}(2dxdy-dz^2) \\
	g_3 &= kdx^2+e^{x\sqrt{4-\omega^2}}\left(
				2dxdy-\frac4{\omega^2}\cos^2(\tfrac{\omega x}{2})dz^2
			\right)\,,\quad \omega\neq 0
\end{align*}
None of these metrics is of Levi-Civita type.
\end{remark}

The following example is based on a 2-dimensional metric of non-constant
curvature.
\begin{example}
	Consider the metric $h=e^{(\beta+2)x}dx^2+e^{\beta x}dy^2$, which has the
	homothetic vector fields
	\[
		u = k_0\,\left( \partial_x+y\partial_y \right)
			+k_1\,\partial_y\,.
	\]
	The metric $g=h+dz^2$, on the other hand, admits the projective symmetry
	algebra
	\[
		u = k_0\,\left( \partial_x+y\partial_y \right)
			+k_1\,\partial_y
			+k_2\,z\partial_z
			+k_3\,\partial_z\,.
	\]
\end{example}


\subsection{Characterisation of \texorpdfstring{$\zeta(z)$}{zeta} for metrics
\texorpdfstring{\eqref{eqn:g1.2-1.for.discussion}}{} admitting projective
vector fields}\label{sec:1D.component}

The purpose of the current section is to find the functions $\zeta(z)$ that
can appear in \eqref{eqn:g1.2-1.for.discussion}, assuming the metric is of
non-constant curvature and admits non-trivial projective vector fields. The main outcome of this section is an ODE for $\zeta(z)$ (see Equation \eqref{eqn:zeta.ode} below), which holds
under mild hypotheses.
%
We start by considering the very special case when the polynomial $\solpol_A$
(see \eqref{eq:solodovnikov.polynomial}) vanishes.

\begin{lemma}\label{la:zeta.ode}
	Let $g$ be the metric \eqref{eqn:g1.2-1.for.discussion} with $\zeta'(z)\ne0$ and of non-constant curvature. Assume $v$ is a projective vector field of $g$.
	\begin{enumerate}
		\item
			If $\solpol_A=0$ (recall Definition \ref{def:solpol}), then $A=0$.
		\item
			If $\solpol_A\ne0$, then $\zeta(z)$ satisfies the ODE
			\begin{equation}\label{eqn:zeta.ode}
				\zeta''=\frac{3b\zeta+C-4B}{2\zeta\,(b\zeta+C-2B)}{\zeta'}^2\,,
			\end{equation}
			where $C$ is as in \eqref{eqn:C} and $B$ as in \eqref{eqn:B}.
	\end{enumerate}
\end{lemma}
\begin{proof}
	For the first part of the claim, $\solpol_A=0$ implies that $b=0$, $a=d$ and $c=0$.
	Because of	Equation~\eqref{eqn:solodovnikov.21}, $\alpha(z)=0$.
	Using \eqref{eqn:alpha.prime} we conclude $C=0$, and thus $d=-3a$ because	$b=0$. Combining this with $a=d$, it follows that $a=d=0$ and thus $A=0$.
%
%
	
	For the second part of the claim, since $\zeta'(z)\ne0$, we can solve \eqref{eqn:alpha.zeta.prime} for $\alpha(z)$. Resubstituting this expression into~\eqref{eqn:alpha.prime}, and reorganising the terms, we obtain\eqref{eqn:zeta.ode}. Note that the denominator does not identically vanish
because of \eqref{eqn:solodovnikov.21.Z}, keeping in mind that $\zeta(z)$ cannot be zero.
\end{proof}



We conclude the section with the following lemma.
\begin{lemma}
	Let $g$ be a [2-1] Levi-Civita metric~\eqref{eqn:g1.2-1.for.discussion} 
	that is not of constant curvature.
	Then $v=\alpha(z)\partial_z$ is a projective vector field for $g$ if
	$\alpha(z)$ and $\zeta(z)$ satisfy the set of
	equations~\eqref{eqn:alpha.prime} and~\eqref{eqn:alpha.zeta.prime}.
\end{lemma}
\begin{proof}
	Note that if $u=0$ holds in~\eqref{eq:raw.degLC.v}, 
	then~\eqref{eqn:Luh.Ch} holds with $C=0$.
	Consequently, \eqref{eqn:alpha.prime} and \eqref{eqn:alpha.zeta.prime}
	become
	\[
		\alpha' = -\frac{b}{2}\,\zeta
		\quad\text{and, respectively,}\quad
		\alpha\zeta'=-b\zeta^2+2B\zeta\,.
	\]
	These two equations are thus equivalent to~\eqref{eq:Lvg1} given the
	conditions $u=0$ and $C=0$.
\end{proof}

\subsection{Levi-Civita metrics of type [2-1] with projective symmetries}

We now aim to describe all Levi-Civita metrics of type [2-1] that have
projective symmetries and do not fall into the situation of
Corollary~\ref{cor:constant.Z}, i.e.\ such that $\zeta'(z)\ne0$.

\begin{lemma}\label{la:zero.dimensional.algebra}
	Let $h$ be a $2$-dimensional metric and let $g=\zeta(z)(h+dz^2)$ where
	$\zeta'\ne0$.
	Assume that the vector field $v=\alpha(z)\partial_z$ is a non-zero
	projective vector field of $g$.
	Then one of the following is realised after a coordinate transformation
	of the type~\eqref{eqn:allowed.changes.coord}.
	\begin{enumerate}
		\item\label{item:zero.dimensional.algebra.exp}
		$\zeta(z)=\varepsilon e^{\beta z}$ and $\alpha=k$.
		\item\label{item:zero.dimensional.algebra.1/z2}
		$\zeta(z)=\frac{\beta}{z^2}$ and $\alpha(z)=\frac{k}{z}$
		\item\label{item:zero.dimensional.algebra.tan}
		$\zeta(z)=\beta(1+\tan^2(\xi z))$ and $\alpha=k\,\tan(\xi
		z)$
		\item\label{item:zero.dimensional.algebra.tanh}
		$\zeta(z)=\beta(1-\tanh^2(\xi z))$ and $\alpha=k\,\tanh(\xi
		z)$
	\end{enumerate}
	where $\xi,\beta\ne0$ and $k\in\R$ and $\varepsilon=\pm1$, at least after 
	a
	constant translation of $z$.
\end{lemma}
\begin{proof}
	We integrate~\eqref{eqn:alpha.prime} and~\eqref{eqn:alpha.zeta.prime}
	under the hypotheses of the statement. Since $u=0$ 
	in~\eqref{eq:raw.degLC.v}, we have $C=0$ and thus
	\[
	\alpha'(z) = \mu\zeta\qquad\text{and}\qquad
	\alpha(z)\zeta'(z) = 2\mu\zeta^2(z)+\eta\zeta
	\]
	with $\mu=-\frac{b}{2}$ and $\eta=2B-C$.
	
	\noindent First assume $\mu=0$ (i.e., $b=0$). Then $\alpha(z)=k\ne0$ and
	so $\zeta'=\frac{\eta}{k}\zeta$. We infer $\eta\ne0$ and 	
	$\zeta(z)=e^{\beta z}$ with $\beta=\frac{\eta}{k}$.
	This is case~\ref{item:zero.dimensional.algebra.exp} of the claim.
	Next assume $\mu\ne0$ (i.e., $b\ne0$).
	Then, substituting $\zeta(z)=\mu^{-1}\,\alpha'(z)$,
	\[
	\alpha(z)\alpha''(z)
	= 2\,\alpha'(z)^2+\eta\alpha'(z)\,.
	\]
	If $\eta=0$, we have $\alpha(z)=\frac{k}{z}$ and obtain
	$\zeta\propto\frac1{z^2}$. This is 
	case~\ref{item:zero.dimensional.algebra.1/z2} of the claim.
	If $\eta\ne0$, let $f(z)=\frac{\alpha(z)}{\eta}$ and obtain
	\[
	f(z)f''(z) = 2\,f'(z)^2+f'(z)\,.
	\]
	which has the two solutions
	\[
	f(z)=\frac{\tan\lb\frac{c_1z}{2}\rb}{c_1}
	\qquad\text{and}\qquad
	f(z)=\frac{\tanh\lb\frac{c_1z}{2}\rb}{c_1}\,.
	\]
	These solutions yield the remaining two cases of the claim. We find
	\[
	\zeta\propto 1+\tan(\xi z)^2\,,\quad\alpha\propto\tan(\xi z)
	\qquad\text{or}\qquad
	\zeta\propto 1-\tanh(\xi z)^2\,,\quad\alpha\propto\tanh(\xi z)
	\]
	where $\xi\ne0$.
\end{proof}
Since in Lemma~\ref{la:zero.dimensional.algebra} the metric $h$ is fixed, the
allowed
transformation \eqref{eqn:allowed.changes.coord} are those with $k_1\in\{\mp 
1\}$.

\begin{lemma}\label{la:one.dimensional.algebra}
	Let $g=\zeta(z)(h+dz^2)$ where $\zeta'\ne0$.
	Assume that $h$ admits a homothetic algebra of dimension
	$\leq1$. Assume, too, that the projective algebra of $g$ is at
	least 2-dimensional,  $\dim\projalg(g)\geq2$.
	Then $h$ admits a 1-dimensional algebra of homothetic vector fields, generated by a non-vanishing vector field $u$, and after a
	translation of $z$ we obtain
	\begin{enumerate}
		\item if $\zeta(z)=\varepsilon e^{\beta z}$ and $u$ is Killing, then
		the projective symmetries of $g$ are $v=k_0u+k_1\partial_z$.
		\item if $\zeta(z)=\frac{\beta}{z^2}$ and $u$ is properly homothetic with $\lie_uh=-Ch$, $C\ne0$, then the projective symmetries of $g$ are
		$v = k_0\,\left(u-\frac{C}{2}\,z\partial_z\right)+\tfrac{k_1}{z}\partial_z$.
		\item if $\zeta(z)=\beta(1+\tan^2(\xi z))$ and $u$ is Killing, then the projective symmetries of $g$ are
		$v = k_0u+k_1\,\tan(\xi z)$, $\xi\in\R\setminus\{0\}$.
		\item if $\zeta(z)=\beta(1-\tanh^2(\xi z))$ and $u$ is Killing, then the projective symmetries of $g$ are $v = k_0u+k_1\tanh(\xi z)$, $\xi\in\R\setminus\{0\}$.
	\end{enumerate}
	for constants $\beta\ne0$ and $\varepsilon\in\{\pm 1\}$ and the constant 
	$C$
	with $\lie_uh=-Ch$.
\end{lemma}
\begin{proof}
	Let $h$ have no projective symmetries, except for the trivial $u=0$.
	Then $g$ has a non-zero projective vector field only if $\zeta$ and
	$\alpha$ are as in Lemma~\ref{la:zero.dimensional.algebra}.
	In these cases, the projective algebra is 1-dimensional. Otherwise, it is
	0-dimensional.
	Thus let $h$ admit a $1$-dimensional homothetic algebra, generated by
	a vector field $u$ with $\lie_uh=-Ch$.
	We prove the assertion by contradiction. Assume that there is
	$\zeta(z)\ne0$ that admits a higher dimensional algebra of projective
	symmetries. We conclude that there are two independent solutions
	$\alpha_1(z),\alpha_2(z)$, satisfying~\eqref{eqn:alpha.zeta.ODEs} for
	the same $\zeta(z)$ and $\lambda$. From the linearity of
	\eqref{eqn:alpha.zeta.ODEs} w.r.t.~$\alpha$ it follows that the function
	$\alpha(z):=\alpha_2(z)-\alpha_1(z)\ne0$ satisfies
	\begin{align*}
		(\alpha\zeta''-\alpha'\zeta')\,\zeta &= \alpha{\zeta'}^2
		\\
		(-2\zeta\alpha''+\alpha\zeta''+\alpha'\zeta' )\,\zeta
		&= -\alpha{\zeta'}^2
	\end{align*}
	By Remark~\ref{rmk:partial.integration.alpha.zeta} it follows that
	\[
		b\zeta(z)=-2\alpha'(z)\qquad\text{and}\qquad
		\alpha(z)\zeta'(z)=-b\zeta^2(z)+2\eta\zeta(z)\,.
	\]
	This is the system solved in
	Lemma~\ref{la:zero.dimensional.algebra}, and so $\zeta(z)$ has to
	be one of
	the solutions in the list. In order to have a second projective
	vector field, independent of $v=u+\alpha(z)\partial_z$, we need to be able
	to find another projective vector field of the form $\bar
	v=u+\bar\alpha(z)\partial_z$, linearly independent of $v$.
	Indeed, this is not possible for all of the metrics of
	Lemma~\ref{la:zero.dimensional.algebra}, yet for some examples it
	is. In order to	find these metrics $g$, i.e.~the functions $\zeta$,
	such that desired $\bar	v$ exists, we solve
	Equations~\eqref{eqn:alpha.prime} and~\eqref{eqn:alpha.zeta.prime}, which
	in terms of~$\bar\alpha(z)$ read
	\begin{subequations}
	\begin{align}
		-2\bar\alpha' &= C+b\zeta
		\label{eqn:baralpha.prime} \\
		\bar\alpha\zeta' &= -b\zeta^2 + \eta\zeta
		\label{eqn:baralpha.zeta.prime}
	\end{align}
	\end{subequations}
	Note that in these equations $\zeta(z)$ is given explicitly, possibly
	involving parameters. We seek a solution $\bar\alpha$. Since
	$\zeta'(z)\ne0$, we have $\bar\alpha=\frac{-b\zeta^2+\eta\zeta}{\zeta}$
	due to~\eqref{eqn:baralpha.zeta.prime}.
	Resubstituting into Equation~\eqref{eqn:baralpha.prime}, we obtain a
	condition on $C$, $b$ and $\eta$. Explicitly, for the cases of
	Lemma~\ref{la:zero.dimensional.algebra} we find
	\begin{enumerate}
		\item If $\zeta=\varepsilon e^{\beta z}$, we solve 
		\eqref{eqn:baralpha.zeta.prime} for
		$\bar\alpha =	\frac{-b\varepsilon}{\beta} e^{\beta z}
						+\frac{\eta}{\beta}$.
		Next, from \eqref{eqn:baralpha.prime}, we obtain
		$\varepsilon b e^{\beta z}-C=0$.
		We conclude that $C=0$, and thus $u$ is a Killing vector field. Moreover, we conclude $b=0$, i.e.\ $v$ is homothetic.
		Finally, we compute $\bar\alpha = \frac{\eta}{\beta}$.
		\item If $\zeta=\frac{\beta}{z^2}$, we find
		$\bar\alpha=\frac{b\beta}{2z}-2\eta z$ and the condition
		$C=4\eta$. Thus, $\bar\alpha=\frac{b\beta}{2z}-\frac{Cz}{2}$.
		\item If $\zeta=\beta(1+\tan^2(\xi z))$, we find analogously that
		$\bar\alpha
		=-\frac12\,\frac{b\beta\tan(\xi z)^2+b\beta-\eta}{\xi\tan(\xi z)}$
		and the condition
		$\lb b\beta+C-\eta \rb + \frac{ b\beta-\eta }{ \tan(\xi z)^2 } = 0$
		and so $\eta=C+b\beta$ and $\eta=b\beta$, implying $C=0$.
		We conclude
		$\bar\alpha = -\frac12\,\frac{b\beta}{\xi}\,\tan(\xi z)$.
		\item The remaining case follows analogously.
	\end{enumerate}
\end{proof}

We are now able to formulate an explicit description of Levi-Civita metrics
of [2-1] type that admit projective vector fields. We begin with the case
when the projective algebra is 1-dimensional.

\begin{proposition}\label{prop:1D.algebra.21}
	Let $g=\zeta(z)(h+dz^2)$ be a Levi-Civita
	metric~\eqref{eqn:g1.2-1.for.discussion} where
	$\zeta(z)$ is not a constant and~$g$ is not of constant curvature.
	If $g$ admits a projective algebra of dimension exactly $1$, then around
	almost every point and locally up to a change of coordinates $g$ falls
	into one of the following cases.
	\begin{enumerate}[label=(\arabic*)]
		\item
		$g = \zeta(z)\,(h+dz^2)$ where $h$ has an exactly 1-dimensional
		Killing algebra generated by the vector field $u$ and where
		\[
			\zeta(z)\notin\left\{
					\eta e^{\beta(z+z_0)}, \frac{\eta}{(z+z_0)^2},
					\eta(1+\tan^2(kz+z_0)), \eta(1-\tanh^2(kz+z_0)) :
					\eta,\beta,k,z_0\in\R
				\right\}
		\]
		Then $\projalg(g)=\la u\ra$ is Killing.
		\item
		$g = \pm e^{\beta z}\,(h+dz^2)$, $\beta\ne0$, where $h$ has no
		Killing vector field.\\
		Then $\projalg(g)=\la\partial_z\ra$ is homothetic.
		\item
		$g = \frac{\eta}{z^2}\,(h+dz^2)$, $\eta\ne0$, where $h$ has no
		homothetic vector field.\\
		Then $\projalg(g)=\la\frac1z\partial_z\ra$ is essential.
		\item
		$g = \eta(1+\tan^2(z))\,(h+dz^2)$, $\eta\ne0$, where $h$ has no
		Killing vector field.\\
		Then $\projalg(g)=\la\tan(z)\partial_z\ra$ is essential.
		\item
		$g = \eta(1-\tanh^2(z))\,(h+dz^2)$, $\eta\ne0$, where $h$ has no
		Killing vector field.\\
		Then $\projalg(g)=\la\tanh(z)\partial_z\ra$ is essential
	\end{enumerate}		
\end{proposition}
\begin{proof}
	The claim follows directly from Lemmas~\ref{la:zero.dimensional.algebra}
	and~\ref{la:one.dimensional.algebra}.
\end{proof}

For Levi-Civita metrics of type [2-1] that admit a projective algebra of
dimension at least~$2$, we then arrive at the following theorem.

\begin{theorem}\label{thm:2.1}
	Let $g=\zeta(z)(h+dz^2)$ be a Levi-Civita
	metric~\eqref{eqn:g1.2-1.for.discussion} where
	$\zeta(z)$ is not a constant and $g$ is not of constant curvature.
	If $g$ admits a projective algebra of dimension at least $2$, then around
	almost every point and locally up to a change of coordinates $g$ falls
	into one of the following cases.
	\begin{enumerate}[label=\alph*)]
		\item\label{item:21a} The metric~$h$ has constant curvature.
		\item\label{item:21b} The homothetic algebra of metric~$h$ is exactly
		of dimension 1. Then there exist local coordinates such that
		\begin{enumerate}[label=(\arabic*)]
			\item
			$g = \varepsilon e^{\beta z}\,(h+dz^2)$, $\beta\ne0$, with
			projective vector fields
			\[ v = k_0\partial_x+k_1\partial_z\,, \] 
			$k_i\in\R$, where $\partial_x$ is a Killing vector field of $h$.
			\item
			$g = \frac{\eta}{z^2}\,(h+dz^2)$ with projective vector fields
			\[ v = k_0\,\left(
					\partial_x-\frac{C}{2}\,z\partial_z
				\right)
				+\frac{k_1}{z}\partial_z\,,
			\]
			$k_i\in\R$, where $\partial_x$ is a properly homothetic vector 
			field of $h$ with $\lie_{\partial_x}h=h$.
			\item
			$g = \eta\,(1+\tan^2(z))\,(h+dz^2)$ with projective	vector fields 
			\[
				v = k_0\partial_x+k_1\tan(z)\partial_z\,,
			\]
			$k_i\in\R$, where $\partial_x$ is a Killing vector field of $h$.
			\item
			$g = \eta\,(1-\tanh^2(z))\,(h+dz^2)$ with projective vector 
			fields
			\[
				v = k_0\partial_x+k_1\tanh(z)\partial_z\,,
			\]
			$k_i\in\R$, where $\partial_x$ is a Killing vector field of $h$.
		\end{enumerate}
		\item\label{item:21c} The homothetic algebra of metric~$h$ is exactly
		of dimension 2. 	
		Locally there exist coordinates such that $h$ can be brought
		into the form
		\begin{equation}\label{eqn:1a-2a}
			h =\varepsilon_1e^{(\beta+2)x} dx^2 + \varepsilon_2e^{\beta x}dy^2
		\end{equation}
		with $\beta\in\R\setminus\{-2,0\}$ and $\varepsilon_i\in\{\pm1\}$.
		Then there exist local coordinates such that $g$ assumes one of the 
		following forms.
		\begin{enumerate}[label=(\arabic*)]
			\item\label{item:Thm.21.c1} 
				For $\eta\in\R,\eta\ne0$, the metric
				$g=\frac{\eta}{z^2}\,(h+dz^2)$,
				has the projective vector fields
				\[
					v = k_0\partial_y
						+ k_1\left(
							2\partial_x
							+2y\partial_y
							+(\beta+2)z\partial_z
						\right)
						+\tfrac{k_2}{z}\partial_z\,,
				\]
				where $k_i\in\R$.
			\item\label{item:Thm.21.c2} 
				The metric $g = \eta e^{z}(h+dz^2)$,
				$\eta\in\R,\eta\ne0$, has the projective vector fields
				\[
					v = k_0\partial_y+k_2\partial_z
				\]
				which are homothetic, with $k_i\in\R$.
			\item\label{item:Thm.21.c3} 
				For $k\ne-2$ the metric $g=\varepsilon\,|z|^{k}\,(h+dz^2)$,
				$\varepsilon\in\{\pm1\}$
				admits the projective vector fields
				\[
					v = k_0\partial_y
						+k_1\left(
							2\partial_x
							+2y\partial_y
							+(\beta+2)z\partial_z
						\right)\,,
				\]
				which are homothetic, with $k_i\in\R$.
			\item\label{item:Thm.21.c4} 
				The metric $g = \eta\,(1+\tan^2(z))\ (h+dz^2)$,
				$\eta\in\R,\eta\ne0$,
				has the projective vector fields
				\[
					v = k_0\partial_y
						+k_2\tan(z)\partial_z\,,
				\]
				with $k_i\in\R$.
			\item\label{item:Thm.21.c5} 
				The metric $g = \eta\,(1-\tanh^2(z))\,(h+dz^2)$,
				$\eta\in\R,\eta\ne0$,
				has the projective vector fields
				\[
					v = k_0\partial_y
						+k_2\tanh(z)\partial_z\,,
				\]
				with $k_i\in\R$.
			\item\label{item:Thm.21.c6} 
				For $k\in\R$, $k\ne-1$, the metric
				$g=\eta\psi'(z)\,(h+dz^2)$,
				$\eta\ne0$,	where
				\[
					(\psi(z)-z)\psi''(z) = 2\psi'(z)\,(\psi'(z)-k),
				\]
				has the projective vector fields
				\[
					v = k_0\partial_y
						+k_1\,(
								2\partial_x
								+2y\partial_y+(\beta+2)(z-\psi(z))\partial_z )
				\]
				with $k_i\in\R$.
		\end{enumerate}
	\end{enumerate}
\end{theorem}

\noindent In Section~\ref{sec:normal.forms} we are going to give explicit
normal forms also for the case~\ref{item:21a}, which are omitted here for
conciseness.

\subsection{Proof of Theorem \ref{thm:2.1}}\label{sec:normal.forms}

\subsubsection{Proof of cases \ref{item:21a} and \ref{item:21b}}
By hypothesis the metric $g$ has a projective algebra of dimension $\geq2$.
Therefore, due to Proposition~\ref{prop:proj.symmetries.descend.21type} and
keeping in mind the considerations within the proof of
Lemma~\ref{la:one.dimensional.algebra}, we conclude that $h$ has a
homothetic algebra of at least dimension~1.
This leaves us with three distinct situations: If $h$ has a homothetic
algebra of exactly dimension~$1$, it must be one of the metrics listed in
Lemma~\ref{la:one.dimensional.algebra}.
If it is of constant curvature, then it has a homothetic algebra of maximal
dimension~$4$. Two-dimensional metrics with a projective algebra of 
dimension~$2$ or~$3$ are
classified in~\cite[Theorem~1]{bryant_2008}. It is then easy to verify that
only~\eqref{eqn:1a-2a} admits a homothetic algebra of dimension $2\leq n\leq
3$.
This leaves us with the following distinct cases:
\begin{enumerate}[label=\alph*)]
	\item
		If $h$ has constant curvature, then there is a coordinate
		transformation~\eqref{eqn:allowed.changes.coord}, see
		Lemma~\ref{la:normalise.curvature}, such that exactly one of the
		following cases occurs:
		\begin{enumerate}[label=\arabic*)]
			\item
			If $h$ has zero curvature, it is locally flat, $h=dx^2+dy^2$.
			\item
			If $h$ has positive constant curvature, it is locally the sphere,
			$h=dx^2+\sin^2(x)dy^2$.
			\item
			If $h$ has negative constant curvature, it is locally of the form
			$h=dx^2+\sinh^2(x)dy^2$.
		\end{enumerate}
	\item
		If $h$ has a homothetic algebra of dimension~$1$, then after a 
		coordinate transformation it is either of the form
		\[
			h = h_{11}(y)dx^2+2h_{12}(y)dxdy+h_{22}(y)dy^2
		\]
		or of the form
		\[
			h = e^{x}\,\left(
				h_{11}(y)dx^2+2h_{12}(y)dxdy+h_{22}(y)dy^2
			\right)\,.
		\]
		In the first case $\partial_x$ is Killing, in the latter it is 
		properly homothetic with $\lie_{\partial_x}h=h$.
	\item
		If $h$ belongs to neither of the previous cases, then it can be
		brought into the form~\eqref{eqn:1a-2a}.
\end{enumerate}
For each of these cases, we need to integrate~\eqref{eqn:alpha.zeta.ODEs} for
$\zeta(z)$ and $\alpha(z)$, where $C$ is determined by the homothetic vector
fields $u$ of $h$.
In fact, see Remark~\ref{rmk:partial.integration.alpha.zeta}, this task is
equivalent to integrating~\eqref{eqn:alpha.prime}
and~\eqref{eqn:alpha.zeta.prime} given $C$ and for suitable $b,B\in\R$.
The parameter choices can be reduced using Lemma~\ref{la:normalising.Lw}.

Cases \ref{item:21a} and \ref{item:21b} of Theorem~\ref{thm:2.1} are
therefore proven, after a suitable
transformation~\eqref{eqn:allowed.changes.coord} of the normal forms
in case \ref{item:21b}.
It remains to conclude the proof in the case of~\eqref{eqn:1a-2a}.

\subsubsection{Proof of case \ref{item:21c}}\label{sec:proof.21}

Note that~\eqref{eqn:1a-2a} admits the homothetic algebra parametrised by
\begin{equation*}
	u = k_0\partial_y + k_1\,(\partial_x+y\partial_y)\,,\quad
	k_i\in\R\,.
\end{equation*}
To proceed, we make use of Lemma~\ref{la:normalising.Lw}, which allows us
to find all projective symmetries of $g$ by integrating the following cases.
By a direct computation we find
\[
	C = -(\beta+2)k_1\,.
\]

\begin{lemma}\label{la:solution.alpha.zeta}
	Let $g=\zeta(z)\,(h+dz^2)$ as in Theorem~\ref{thm:2.1}. Assume $g$ admits
	a projective algebra larger than the one generated by $\partial_y$.
	\begin{enumerate}
		\item\label{item:solution.21.1}
		If $g$ admits a Killing vector field in addition to $\partial_y$, then after a
		transformation~\eqref{eqn:allowed.changes.coord} and multiplying $g$
		with a constant, we have the metric
		\[
			g= \frac1{z^2}(h+dz^2)\,,
		\]
		which admits also the projective vector
		field~$v=\partial_x+y\partial_y-\frac{\beta+2}{2}\,z\partial_z$.
		\item\label{item:solution.21.2}
		If $g$ admits a properly homothetic vector field in addition to the 
		Killing vector field~$\partial_y$, then, after a
		transformation~\eqref{eqn:allowed.changes.coord} and multiplying~$g$
		with a constant, either we have the metric
		\[
			g= z^k\,(h+dz^2)
		\]
		($k\ne0$), which admits also the projective vector
		field~$v=\partial_x+y\partial_y-\frac{\beta+2}{2}\,z\partial_z$, or 
		we have the metric
		\[
			g= e^{kz}\,(h+dz^2)
		\]
		($k\ne0$), which admits also the projective vector 
		field~$v=\partial_z$.
		\item\label{item:solution.21.3}
		If $g$ admits an essential projective vector field in addition to the 
		Killing vector field~$\partial_y$, then, after a
		transformation~\eqref{eqn:allowed.changes.coord} and multiplying $g$
		with a constant, one of the following is attained
		\begin{itemize}
			\item
			$g=\frac1{z^2}\,(h+dz^2)$ with projective vector field
			$v=\frac1z\,\partial_z$.
			\item
			$g=(1+\tan^2(z))\,(h+dz^2)$ with projective vector
			field $v=\tan(z)\,\partial_z$.
			\item
			$g=(1-\tanh^2(z))\,(h+dz^2)$ with projective vector
			field $v=\tanh(z)\,\partial_z$.
			\item
			$g=\psi'(z)\,(h+dz^2)$ where
			\begin{equation}\label{eqn:psi.ode}
				(\psi(z)-z)\,\psi''(z) = 2\psi'(z)\,(\psi'(z)-k)\,,
				\qquad k\in\R\,,
			\end{equation}
			and with the essential projective vector field
			$v=\partial_x+y\partial_y+\frac12\,(z-\psi(z))\,\partial_z$.
		\end{itemize}
	\end{enumerate}
\end{lemma}
\begin{proof}
	We begin with part~\ref{item:solution.21.1} of the claim and compute 
	Killing vector fields $v$ for $g$ by setting~$a=0,b=0$	(and thus 
	$C=B=d$) 
	in \eqref{eqn:alpha.prime}	and~\eqref{eqn:alpha.zeta.prime}. We hence 
	need to solve
	\begin{align*}
		-2\alpha'(z) &= C \\
		\alpha(z)\zeta'(z) &= C\zeta(z)\,.
	\end{align*}
	We begin by assuming $C\ne0$. After a suitable 	
	translation~\eqref{eqn:allowed.changes.coord}, we arrive at the first 
	case.
	If $C=0$, on the other hand, we immediately obtain $\alpha(z)=0$, for any
	non-constant function~$\zeta$. We can omit this solution because its
	projective algebra is spanned by $\partial_y$ and thus 1-dimensional.
	
	We continue with part~\ref{item:solution.21.2} of the claim.
	Properly homothetic vector fields $v$ are obtained assuming $a=1,b=0$
	($C=3+d,B=1+d=C-2$). We need to integrate
	\begin{align*}
		-2\alpha'(z) &= C \\
		\alpha(z)\zeta'(z) &= (C-4)\zeta(z)\,.
	\end{align*}
	If $C\ne0$, then after a suitable
	transformation~\eqref{eqn:allowed.changes.coord},
	\begin{equation*}
		\alpha(z) = -\frac{C}{2}\,z\,,\qquad
		\zeta(z) = C_1 z^{2\lb \frac4C-1\rb}
	\end{equation*}
	where $C_1\in\R$.
	If $C=0$, we find the solution
	\[
		\alpha(z)=C_0\ne0\qquad\text{for}\qquad
		\zeta(z)=C_1\,e^{-\frac4{C_0}z}\,.
	\]
	
	Finally, we consider part~\ref{item:solution.21.3} of the claim.
	Essential vector fields $v$ are obtained assuming $a=0,b=1$ ($C=2\rho+d$, 
	$B=d$). We need to integrate
	\begin{align*}
		-2\alpha'(z) &= C+\zeta(z) \\
		\alpha(z)\zeta'(z) &= -\zeta^2(z)+(2B-C)\zeta(z)
	\end{align*}
	We solve the first equation for $\zeta$ and substitute the result
	into the second condition, obtaining
	\begin{equation}\label{eqn:alpha.ode.essential}
		\zeta=-2\alpha'-C\qquad\text{where}\quad
		-2\alpha\alpha''
		= -(2\alpha'+C)^2-(2B-C)(2\alpha'+C)
	\end{equation}
	We begin with the case $C=0$.
	The ODE becomes
	\[
	\alpha\alpha''=2(\alpha')^2+2B\alpha'
	\]
	If $B=0$, we have $\alpha=\frac{C_2}{C_1+z}$ and thus $\zeta=\frac1{z^2}$
	up to a constant translation of $z$ and a multiplication of~$\zeta$ with
	a constant.
	Thus assume $B\ne0$.
	Introducing $\psi=\tfrac{\alpha}{2B}$ the ODE translates into
	\[
		\psi(z)\psi''(z) = 2\psi'(z)^2+\psi'(z).
	\]
	with the solutions
	\[
		\psi(z)=\frac{1}{C_1}\tan\lb\frac{C_1z}{2}\rb
		\quad\text{and}\quad
		\psi(z)=\frac{1}{C_1}\tanh\lb\frac{C_1z}{2}\rb\,.
	\]
	We thus have
	\[
		\alpha(z) = \frac{2B}{C_1}\tan(\tfrac{C_1z}{2})
	\]
	and, up to multiplication by a constant factor,
	\[
		\zeta(z) = (1+\tan^2(C_1z))\quad\text{or}\quad
		\zeta(z) = (1-\tanh^2(C_1z))\,,
	\]
	from which the claim is easily obtained.

	Finally, consider $C\ne0$. In this case it is helpful to introduce
	\begin{equation*}
		\psi(z) = \frac{2\alpha}{C}+z\,,
	\end{equation*}
	which turns~\eqref{eqn:alpha.ode.essential} into the
	ODE~\eqref{eqn:psi.ode} with $k=1-\frac{2B}{C}$.
\end{proof}

For the cases~\ref{item:Thm.21.c1} to~\ref{item:Thm.21.c5} of the claim in 
part~\ref{item:21c} of Theorem~\ref{thm:2.1}, the proof is a direct 
consequence of Lemma~\ref{la:solution.alpha.zeta}.
Case~\ref{item:Thm.21.c6} of part~\ref{item:21c} is also a direct consequence 
of Lemma~\ref{la:solution.alpha.zeta}, but we need to exclude all the 
solutions that are equivalent, under~\eqref{eqn:allowed.changes.coord}, to 
one of the previous cases.
We can check this by assuming that there is, for some $k\in\R$, a solution
$\psi$ of~\eqref{eqn:psi.ode}, such that the metric has the form
$\psi'(z)\,(h+dz^2)$, with $h$ as in~\eqref{eqn:1a-2a}.
For the metric in case~\ref{item:Thm.21.c1} we have $\psi'(z)=\frac{1}{z^2}$ 
(note that the constant conformal factor $\eta$ is irrelevant).
We thus have $\psi(z)=\frac1z+s$ for some constant $s\in\R$. Resubstituting
into~\eqref{eqn:psi.ode} yields the condition
\[
	(k+1)z-s = 0\,,
\]
which can only be realised if $k=-1$ and $s=0$. On the other hand, if $k=-1$,
the ODE~\eqref{eqn:psi.ode} has the solution
\[
	\psi(z) = -\frac{k_0+k_1z}{k_1+z}\,,
\]
for some $k_0,k_1\in\R$, from which we obtain
\[
	\psi'(z) = \frac{k_0-k_1^2}{(k_1+z)^2}\,.
\]
This corresponds to the metric ($\mu\in\R,\mu\ne0$)
\[
	g = \eta\psi'(z)\,(h+dz^2) = \frac{\mu}{(k_1+z)^2}
\]
and after a transformation~\eqref{eqn:allowed.changes.coord} we obtain the
metric from case~\ref{item:Thm.21.c1} of Theorem~\ref{thm:2.1}. We therefore 
exclude the value $k=-1$ from case~\ref{item:Thm.21.c6}.

It is easily confirmed that none of the cases~\ref{item:Thm.21.c2} 
to~\ref{item:Thm.21.c5} can be realised as a special case 
of~\ref{item:Thm.21.c6}. For example, if $\psi'(z)=e^z$, then $\psi(z)=e^z+s$ 
for some $s\in\R$. Resubstituting into~\eqref{eqn:psi.ode}, we have
\[
	(s+2k-z)e^z-e^{2z} = 0\,,
\]
which cannot be realised for any $s,k$ with generic $z$.
The cases~\ref{item:Thm.21.c3} to~\ref{item:Thm.21.c5} are checked 
analogously.

\subsubsection{Solutions of Equation~\eqref{eqn:psi.ode}}

The last case of Theorem~\ref{thm:2.1} is given implicitly through solutions
of an ODE for a function $\psi(z)$, from which both the metric and the
projective symmetries are obtained.
It is however possible to obtain solutions in more concrete form using the
system~\eqref{eqn:zeta.ode}.
Consider case~\ref{item:Thm.21.c6} of Theorem~\ref{thm:2.1}. A closer 
inspection of the proof,
and a comparison with Lemma~\ref{la:zeta.ode}, shows that instead of
solving~\eqref{eqn:psi.ode}, we can also solve~\eqref{eqn:zeta.ode}. Since we
need to determine $\zeta(z)$ only up to a constant factor, we can
equivalently solve
\[
	\zeta''(z)
	= \frac{3\zeta(z)+2k-1}{2\zeta(z)\,(\zeta(z)+k)}\,\zeta'(z)^2
\]
We claim that $\zeta''(z)\ne0$. Indeed, if $\zeta(z)=c_0z+c_1$, then
\[
	0 = \frac{3c_0z+3c_1+2k-1}{2c_0z+3c_1\,(c_0z+c_1+k)}\,c_0^2\,,
\]
implying $c_0=0$, which contradicts the hypothesis of Theorem~\ref{thm:2.1}.
We conclude that $\zeta''(z)\ne0$.

\begin{example}
	Let us now consider the special case when $k=0$. Since $\zeta''(z)\ne0$,
	we
	introduce $f(\zeta(z))=\zeta'(z)^2$. Thus we need to solve
	\[
		f'(\zeta)
		= \frac{3\zeta-1}{\zeta^2}\,f(\zeta)\,.
	\]
	We obtain
	\[
		\zeta'' = f(\zeta) = c_2 f(\zeta)^3 e^{\frac{1}{f(\zeta)}}\,.
	\]
	This equation can be solved and we obtain, up to a
	transformation~\eqref{eqn:allowed.changes.coord} and up to a constant
	factor,
	\[
		\zeta(z) = \frac{1}{\mathrm{inverf}(z)}\,,
	\]
	where $\mathrm{inverf}$ is the inverse error function.
\end{example}

The cases with $k\ne0$ are generally less easy to solve.
Introducing again $f(\zeta(z))=\zeta'(z)^2$, we arrive at
\[
	f'(\zeta)
	= \frac{3\zeta+2k-1}{\zeta\,(\zeta+k)}\,f(\zeta)\,.
\]
The computer algebra software \emph{Wolfram Mathematica 12.3}
\cite{Mathematica} yields the
solution
\[
	f(\zeta) = c_3 \zeta^{\frac{2k+1}{k}}
					(\zeta-k)^{\frac{k-1}{k}}\,.
\]
Solving $f(\zeta(z))=\zeta''(z)^2$, the software finds
\[
	\zeta(z) = F^{-1}(z)\,,
\]
up to a transformation~\eqref{eqn:allowed.changes.coord} and up to a constant
factor, where $F^{-1}$ denotes the inverse of the function
\[
	F : t \mapsto
	-\frac{
			2k\,\,
			t^{-\frac1{2k}}
				\left(t-k\right)^{\frac1{2k}}
				\left(\frac{k-t}{k}\right)^{\frac{k-1}{2k}}
			\,\,{}_2F_1\left[
				\frac{k-1}{2k},
				-\frac{1}{2k},
				\frac{2k-1}{2k},
				\frac{t}{k}
			\right]
		}{
			\sqrt{k-t}
		}
\]
for the hypergeometric special function ${}_2F_1$.
	
\begin{example}
Let us consider the special case when $k=1$.
The solution to \eqref{eqn:psi.ode} is
\[
	\psi(z)=z-\frac{\tanh( k_0 +k_1z) }{k_1}\,,\,\,k_1\neq
	0\quad\text{or}\quad \psi(z)=z\,,
\]
so that
\[
	\psi'(z) = \tanh^2(k_0 +k_1z)\,,\,\,k_1\neq 0\quad\text{or}\quad
	\psi'(z) = 1\,.
\]
The latter solution is not allowed by the hypotheses of Theorem~\ref{thm:2.1}.
Up to a transformation~\eqref{eqn:allowed.changes.coord}, and up to
multiplication of the metric by a constant factor, this yields
\[
	g = \tanh^2(z)\,(h+dz^2)\,,
\]
where $h$ is given by~\eqref{eqn:1a-2a}. Its projective algebra is indeed
parametrised by
\[
	v = k_0\partial_y
		+ k_1\,(2\partial_x+2y\partial_y+(\beta+2)\tanh(z)\partial_z)\,,
\]
which is an essential projective vector field for any $k_1\ne0$.
\end{example}

\subsection{[2-1]-type Levi-Civita metrics with constant curvature metric
\texorpdfstring{$h$}{h}}\label{sec:cc.h}

In this section we consider case a) of Theorem \ref{thm:2.1}. For brevity we
shall consider only the case when~$h$ has Riemannian signature as for the
other signatures the corresponding results can be obtained in complete
analogy.

\subsubsection{[2-1]-type Levi-Civita metrics with flat metric
\texorpdfstring{$h$}{h}}\label{sec:flat.h}

For for any $\zeta\ne0$, the metric
\begin{equation}\label{eqn:g.zeta.eucl}
g=\zeta(z)(dx^2+dy^2+dz^2)
\end{equation} admits
the Killing vector fields
\[
	v = k_0(y\partial_x-x\partial_y)
		+k_1\partial_x
		+k_2\partial_y\,,
\]
$k_i\in\R$, which straightforwardly arise from Killing vector fields of $h$
for any $\zeta(z)\ne0$. The metric $h$ also admits the properly homothetic
vector field
\[
	u = k_3\,(x\partial_x+y\partial_y)\,.
\]
If $g$ admits a larger projective algebra, then $g$ assumes at least one of
the following forms after a suitable coordinate transformation.
\begin{enumerate}
	\item
	For $k(k-2)\ne0$, the metric
	$g = \varepsilon |z|^{-k}\ (dx^2+dy^2+dz^2)$, $\varepsilon\in\{\pm1\}$,
	has the projective vector fields
	\[
		v = k_0(y\partial_x-x\partial_y)
			+k_1\partial_x
			+k_2\partial_y
			+k_3(x\partial_x+y\partial_y+z\partial_z)
	\]
	\item
	The metric $g = \varepsilon e^{z}(dx^2+dy^2+dz^2)$,
	$\varepsilon\in\{\pm1\}$,	has the projective vector fields
	\[
		v = k_0(y\partial_x-x\partial_y)
			+k_1\partial_x
			+k_2\partial_y
			+k_4\partial_z\,.
	\]
	\item 
	The metric $g = \eta(1+\tan^2(z))(dx^2+dy^2+dz^2)$, $\eta\ne0$,
	has the projective vector fields
	\[
		v = k_0(y\partial_x-x\partial_y)
			+ k_1\partial_x
			+ k_2\partial_y
			+ k_4\tan(z)\partial_z\,.
	\]
	\item 
	The metric $g = \eta(1-\tanh^2(z))(dx^2+dy^2+dz^2)$,
	$\eta\ne0$,
	has the projective vector fields
	\[
		v = k_0(y\partial_x-x\partial_y)
			+ k_1\partial_x
			+ k_2\partial_y
			+ k_4\tanh(z)\partial_z\,.
	\]
	\item 
	For $k\in\R$, $k\ne-1$, the metric
	$g=\eta\psi'(z)\,(dx^2+dy^2+dz^2)$,
	$\eta\ne0$,	where
	\[
		(\psi(z)-z)\psi''(z) = 2\psi'(z)\,(\psi'-k)\,,
	\]
	has the projective vector fields
	\[
		v = k_0(y\partial_x-x\partial_y)
			+k_1\partial_x
			+k_2\partial_y
			+k_3(x\partial_x+y\partial_y+(z-\psi(z))\partial_z)
	\]
\end{enumerate}

\begin{remark}
The metric \eqref{eqn:g.zeta.eucl} is of constant curvature if and only if $\zeta(z)=\mp \frac{1}{(c_0+c_1z)^2}$ with $c_i\in\R$ such that $c_0^2+c_1^2\neq 0$. There are no metrics of type (5) with constant curvature.
\end{remark}
In order to obtain the above list, we proceed analogously to the proof of
Theorem \ref{thm:2.1}.
Generically the metric $h=dx^2+dy^2$ admits the homothetic vector fields
\[
	u = k_0(y\partial_x-x\partial_y)
		+ k_1\partial_x
		+ k_2\partial_y
		+ k_3(x\partial_x+y\partial_y)\,,
\]
and we easily compute
\[
	C = -2k_3\,.
\]
The admissible functions $\zeta(z)$, i.e.\ the functions that lead to new
projective vector fields, can be found analogously to
Lemma~\ref{la:solution.alpha.zeta}.

First, note that the metrics $g$ with $\zeta\propto z^{\mu}$ are of
constant curvature if and only if $\mu(\mu+2)(\mu-4)=0$, and thus
$\zeta\propto\frac1{z^2}$ as well as $\zeta\propto z^4$ must be omitted. The
remaining metrics are of non-constant curvature. Hence we obtain the first
case in the list, after a suitable rescaling of the coordinates.
Similarly, the metrics with $\zeta\propto e^{-kz}$ are of constant curvature
if and only if $k=0$ and thus we find the second case in the list after a
suitable affine transformation of the coordinates.
Next, the metrics $g$ with $\zeta\propto e^{-\tfrac{\mu}{z}}$ do not
admit additional projective symmetries and can therefore be omitted from the
list.
For $\zeta\propto(1+\tan^2(C_1z))$, we obtain metrics of
non-constant curvature if and only if $C_1\ne0$, from which we infer the
third case of the list, after a suitable rescaling of the coordinates. The
last case follows analogously.

\subsubsection{[2-1]-type Levi-Civita metrics with spherical metric
\texorpdfstring{$h$}{h}}\label{sec:sphere.h}

With exactly the same strategy as in Section~\ref{sec:flat.h} we can also
consider the case when $h$ has positive constant curvature. According to
Lemma~\ref{la:normalise.curvature}, we can restrict to curvature $+1$.

For $\zeta\ne0$, we consider the metric
\begin{equation}\label{eqn:g.zeta.sphere}
g=\zeta(z)\,(dx^2+\sin^2(x)dy^2+dz^2)\,,
\end{equation}
which admits the linearly independent Killing vector fields
\begin{align*}
	u_0 &= \partial_y \\
	u_1 &= \cos(y)\partial_x-\frac{\sin(y)}{\tan(x)}\partial_y \\
	u_2 &= \sin(y)\partial_x+\frac{\cos(y)}{\tan(x)}\partial_y\,,
\end{align*}
that straightforwardly arise from the Killing vector fields of $h$.
If $g$ admits a larger projective algebra, then $g$ assumes at least one of
the following forms after a suitable coordinate transformation.
\begin{enumerate}
	\item
	The metric
	$g = \frac{\eta}{z^2}\,(dx^2+\sin^2(x)dy^2+dz^2)$, $\eta\ne0$,
	has the projective vector fields
	\[
		v = k_0\,u_0+k_1\,u_1+k_2\,u_2
			+\frac{k_3}{z}\partial_z
	\]
	\item
	The metric $g = \varepsilon e^{z}(dx^2+\sin^2(x)dy^2+dz^2)$,
	$\varepsilon\in\{\pm1\}$,	has the projective vector fields
	\[
		v = k_0\,u_0+k_1\,u_1+k_2\,u_2
			+k_3\partial_z\,.
	\]
	\item 
	The metric $g = \eta(1+\tan^2(z))(dx^2+\sin^2(x)dy^2+dz^2)$, $\eta\ne0$,
	has the projective vector fields
	\[
		v = k_0\,u_0+k_1\,u_1+k_2\,u_2
			+k_3\tan(z)\partial_z\,.
	\]
	\item 
	The metric $g = \eta(1-\tanh^2(z))(dx^2+\sin^2(x)dy^2+dz^2)$,
	$\eta\ne0$,
	has the projective vector fields
	\[
		v = k_0\,u_0+k_1\,u_1+k_2\,u_2
			+k_3\tanh(z)\partial_z\,.
	\]
\end{enumerate}
\begin{remark}
The metric \eqref{eqn:g.zeta.sphere} is of constant curvature if and only if $\zeta(z)=\mp \frac{e^{-2z}}{(c_0+c_1e^{-2z})^2}$,  with $c_i\in\R$ such that $c_0^2+c_1^2\neq 0$.
\end{remark}

\subsubsection{[2-1]-type Levi-Civita metrics with hyperbolic metric
\texorpdfstring{$h$}{h}}\label{sec:hyperbolic.h}

We conclude this section with the case when $h$ has negative constant
curvature, i.e., according to Lemma~\ref{la:normalise.curvature}, curvature
$-1$. The strategy is the same as in the previous two cases.
For $\zeta\ne0$, we consider the metric
\begin{equation}\label{eqn:g.zeta.hyperb}
g=\zeta(z)\,(dx^2+\sinh^2(x)dy^2+dz^2)\,,
\end{equation}
which admits the Killing vector fields
\begin{align*}
	u_0 &= \partial_y \\
	u_1 &= \cos(y)\partial_x-\frac{\sin(y)}{\tanh(x)}\partial_y \\
	u_2 &= \sin(y)\partial_x+\frac{\cos(y)}{\tanh(x)}\partial_y
\end{align*}
that straightforwardly arise from the Killing vector fields of $h$.
If $g$ admits a larger projective algebra, then $g$ assumes at least one of
the following forms after a suitable coordinate transformation.
\begin{enumerate}
	\item
	The metric
	$g = \frac{\eta}{z^2}\,(dx^2+\sin^2(x)dy^2+dz^2)$, $\eta\ne0$,
	has the projective symmetries
	\[
		v = k_0\,u_0+k_1\,u_1+k_2\,u_2
			+\frac{k_3}{z}\partial_z
	\]
	\item
	The metric $g = \varepsilon e^{z}(dx^2+\sinh^2(x)dy^2+dz^2)$,
	$\varepsilon\in\{\pm1\}$,	has the projective symmetries
	\[
		v = k_0\,u_0+k_1\,u_1+k_2\,u_2
			+k_3\partial_z\,.
	\]
	\item 
	The metric $g = \eta(1+\tan^2(z))(dx^2+\sinh^2(x)dy^2+dz^2)$, $\eta\ne0$,
	has the projective symmetries
	\[
		v = k_0\,u_0+k_1\,u_1+k_2\,u_2
			+k_3\tan(z)\partial_z\,.
	\]
	\item 
	The metric $g = \eta(1-\tanh^2(z))(dx^2+\sinh^2(x)dy^2+dz^2)$,
	$\eta\ne0$,
	has the projective symmetries
	\[
		v = k_0\,u_0+k_1\,u_1+k_2\,u_2
			+k_3\tanh(z)\partial_z\,.
	\]
\end{enumerate}

\begin{remark}
The metric \eqref{eqn:g.zeta.hyperb} is of constant curvature if and only if $\zeta(z)=\mp \frac{1}{(c_1\sin(z)+c_2\cos(z))^2}$,  with $c_i\in\R$ such that $c_1^2+c_2^2\neq 0$.
\end{remark}

\section{Proof of Theorem~\ref{thm:main}}\label{sec:proof.main}

%
First let us recall that the metrics of Theorem~\ref{thm:main} are of 
non-constant curvature, and by assumption they do not admit any homothetic 
vector field.
The degree of mobility is therefore~$2$, and since the metric is of 
Riemannian signature, it follows that it is a Levi-Civita metric either of 
type [1-1-1] or of type [2-1].
Let us begin with type [1-1-1].
Reviewing Theorem~\ref{thm:1.1.1}, we observe that the metrics in 
Propositions~\ref{prop:111.2constantEVs} and~\ref{prop:111.1constantEVs} do 
admit a Killing vector field and therefore do not fall under the assumptions 
of Theorem~\ref{thm:main}.
The desired metrics thus follow, maybe after an obvious change of 
coordinates, from a straightforward analysis of the metrics in 
Proposition~\ref{prop:111.0constantEVs}. We arrive at the cases 
\ref{item:main.111.1} to~\ref{item:main.111.5} of 
Theorem~\ref{thm:main}.

Hence let us proceed to Levi-Civita metrics of [2-1] type.
Note that the conformal factor $\zeta(z)$ must be non-constant since 
otherwise there necessarily exists a Killing vector field contrary to the 
hypothesis. We are thus left with the metrics covered in 
Theorem~\ref{thm:2.1} and Proposition~\ref{prop:1D.algebra.21}.
The metrics covered by Theorem~\ref{thm:2.1} always admit a homothetic vector 
field, however, and thus we are left with those in 
Proposition~\ref{prop:1D.algebra.21}.
Among these metrics only three cases are compatible with the assumptions of 
Theorem~\ref{thm:main}. We arrive at cases \ref{item:main.21.1} 
to~\ref{item:main.21.3} in the claim.

\section*{Acknowledgements}
The authors thank Vladimir S.~Matveev and Stefan Rosemann for discussions.
Both authors acknowledge support through the project PRIN 2017 ``Real and
Complex Manifolds: Topology, Geometry and holomorphic dynamics'',  the project ``Connessioni, proiettive, equazioni di Monge-Amp\`ere e
sistemi integrabili'' (INdAM), the MIUR grant ``Dipartimenti di Eccellenza 2018-2022 (E11G18000350001)". AV is grateful for funding by the German Research Foundation (Deutsche Forschungsgemeinschaft) through the fellowship project grant 353063958. AV thanks the University of Stuttgart, Germany, and the University of New South Wales Sydney, Australia, for their kind hospitality. The authors are members of GNSAGA of INdAM.

\sloppy
\printbibliography

\end{document}